\documentclass[12pt]{article}

\usepackage{amsmath}
\usepackage{amssymb}
\usepackage{epsfig}
\usepackage{multicol}

\newenvironment{proof}[1][Proof]{\begin{trivlist}
\item[\hskip \labelsep {\bfseries #1}]}{$\blacksquare$ \end{trivlist}}

\newtheorem{ex}{Example}[section]
  \newtheorem{theorem}{Theorem}[section]

  \newtheorem{lemma}[theorem]{Lemma}

  \newtheorem{construction}{Construction}[section]
\newtheorem{problem}{Problem}[section]

\newcommand{\qed}{\nobreak \ifvmode \relax
 \else \ifdim\lastskip<1.5em \hskip-\lastskip
\hskip1.5em plus0em minus0.5em \fi \nobreak
 \vrule height0.75em width0.5em depth0.25em\fi}

\begin{document}

\title{Cyclic, Simple, and Indecomposable Three-fold Triple Systems}
\author{
\vspace{0.25in}
 Nabil Shalaby \hspace{0.5in} Bradley Sheppard \hspace{0.5in} Daniela Silvesan\\
 Department of Mathematics and Statistics \\
 Memorial University of Newfoundland \\
 St. John's, Newfoundland \\
 CANADA A1C 5S7
}
 \maketitle
 
\vspace{0.25in}

{\bf This paper is dedicated to the memory of Dr. Rolf Rees (1960-2013).}

\vspace{0.25in}

\begin{abstract}

In $2000$, Rees and Shalaby constructed simple indecomposable two-fold cyclic triple systems for all $v\equiv 0,~1,~3,~4,~7,~\textup{and}~9~(\textup{mod}~12)$
 where $v=4$ or $v\geq12$, using Skolem-type sequences. 

We construct, using Skolem-type sequences, three-fold triple systems having the properties of being cyclic, simple, and indecomposable
 for all admissible orders $v$, with some possible exceptions for $v=9$ and $v=24c+57$, where $c\geq 2$ is a constant.
To prove the simplicity we used a Mathematica computer program. We list in the Appendix the code and the results of the program.
\end{abstract}


\section{Introduction}

A {\em $\lambda$-fold} triple system of order $v$, denoted by TS$_{\lambda}(v)$, is a pair $(V,\cal{B})$ where
$V$ is a $v$-set of points and $\cal{B}$ is a set of $3$-subsets {\em (blocks)} such
that any $2$-subset of $V$ appears in exactly $\lambda$ blocks. An {\em automorphism group} of an TS$_{\lambda}(v)$ is a permutation on $V$ leaving $\cal{B}$ invariant. An TS$_{\lambda}(v)$ is {\em cyclic} if its automorphism group contains a $v$-cycle. If $\lambda=1$, an TS$_{\lambda}(v)$ is called {\em Steiner triple system} and is denoted by STS$(v)$. A cyclic STS$(v)$ is denoted by CSTS$(v)$.

An TS$_\lambda(v)$ is {\em simple} if it contains no repeated blocks. An TS$_\lambda(v)$ is called {\em indecomposable} if its blocks set $\cal{B}$ cannot be partitioned into sets $\mathcal{B}_1$, $\mathcal{B}_2$ of blocks of the form TS$_{\lambda_1}(v)$ and TS$_{\lambda_2}(v)$, where $\lambda_1+\lambda_2=\lambda$ with $\lambda_1,\lambda_2\geq 1$. A cyclic TS$_\lambda(v)$ is called {\em cyclically indecomposable} if its block set $\cal{B}$ cannot be partitioned into sets $\mathcal{B}_1$, $\mathcal{B}_2$ of blocks to form a cyclic TS$_{\lambda_1}(v)$ and TS$_{\lambda_2}(v)$, where $\lambda_1+\lambda_2=\lambda$ with $\lambda_1,\lambda_2\geq 1$.

The constructions of triple systems with the properties cyclic, simple, and indecomposable, were studied by many researchers for one property at a time; for example,
 cyclic triple systems for all $\lambda$s were constructed in \cite{colbourn2, silvesan}, simple for $\lambda=2$
 in \cite{stinson} and simple for every
 $v$ and $\lambda$ satisfying the necessary conditions in \cite{dehon}. Also, some of the properties were combined in studies. For example, in \cite{wang},
 cyclic and simple two-fold triple systems for all admissible orders were constructed, while in \cite{archdeacon, colbournrosa, dinitz, kramer, milici, zhang2}, simple and 
 indecomposable designs for $\lambda=2,3,4,5,6$ and all admissible $v$ were constructed. In \cite{zhang}, simple and indecomposable designs 
were constructed for all $v\geq 24\lambda -5$ satisfying the necessary conditions. For the general case of $\lambda > 6$, Colbourn and Colbourn \cite{colbourn} 
constructed a single indecomposable TS$_{\lambda}(v)$ for each odd $\lambda$. Shen \cite{shen} used Colbourn and Colbourn result and some recursive constructions 
to prove the necessary conditions are asymptotically sufficient. Specifically, if $\lambda$ is odd, then there exists a constant $v_0$ depending on $\lambda$ with an
 indecomposable simple TS$_\lambda(v)$ design for all $v\geq v_0$ satisfying the necessary conditions. In \cite{gruttmuller}, the authors constructed two-fold 
cyclically indecomposable triple systems for all admissible orders. The authors also checked exhaustively the cyclic triple systems TS$_\lambda(v)$ for
 $\lambda=2,\;v\leq 33$ and $\lambda=3,\;v\leq21$ that are cyclically indecomposable and determined if they are decomposable (to non cyclic) or not. 

In $2000$, Rees and Shalaby \cite{reesshalaby} constructed simple indecomposable two-fold cyclic triple systems for all $v\equiv 0,~1,~3,~4,~7,~\textup{and}~9~(\textup{mod}~12)$
 where $v=4$ or $v\geq12$ using Skolem-type sequences. They acknowledged that the analogous problem for $\lambda >2$ is more difficult.

In $1974$, Kramer \cite{kramer} constructed indecomposable three-fold triple systems for all admissible orders. We noticed that Kramer's construction for 
$v\equiv 1\;\textup{or}\;5\,(\textup{mod}\;6)$ gives also cyclic and simple designs.

In this paper, we construct three-fold triple systems having the properties of being cyclic, simple, and indecomposable for all admissible orders
$v\equiv 3~(\textup{mod}~6)$, except for $v=9$ and $v=24c+57$, $c\geq 2$.

\section{Preliminaries}

Let $D$ be a multi set of positive integers with $|D|=n$. A {\em Skolem-type sequence of order $n$} is a sequence $(s_1,\ldots,s_t),t\geq 2n$ of $i\in D$ such that
 for each $i\in D$ there is exactly one $j\in \{1,\ldots,t-i\}$ such that $s_j=s_{j+i}=i$. Positions in the sequence not occupied by integers $i\in D$ contain null
 elements. The null elements in the sequence are also called {\em hooks}, {\em zeros} or {\em holes}. As examples, $(1,1,6,2,5,2,1,1,6,5)$ is a Skolem-type sequence 
of order $5$ and $(7,5,2,0,2,0,5,7,1,1)$ is a Skolem-type sequence of order $4$.

Some special Skolem-type sequences are described below.

A {\it Skolem sequence of order $n$} is a sequence $S_n=(s_{1},s_{2},\ldots, s_{2n})$ of $2n$ integers which satisfies the conditions:
\begin{enumerate}
\item for every $k\in \{1,2,\ldots,n\}$ there are exactly two elements $s_{i},s_{j}\in S$ such that $s_{i}=s_{j}=k$, and
\item if $s_{i}=s_{j}=k,\;i<j$, then $j-i=k.$
\end{enumerate}
Skolem sequences are also written as collections of ordered pairs $\{(a_i,b_i):1\leq i\leq n,\;b_i-a_i=i\}$ with $\cup_{i=1}^{n}\{a_i,b_i\}=\{1,2,\ldots,2n\}$.

For example, $S_5=(1,1,3,4,5,3,2,4,2,5)$ is a Skolem sequence of order $5$ or, equivalently, the collection $\{(1,2),(7,9),(3,6),(4,8),(5,10)\}$.

Equivalently, a {\it Skolem sequence of order $n$} is a Skolem-type sequence with $t=2n$ and $D=\{1,\ldots,n\}$.

A {\em hooked Skolem sequence of order $n$} is a sequence $hS_n=(s_{1},\ldots, s_{2n-1}, s_{2n+1})$ of $2n+1$ integers which satisfies the above definition, as well as $s_{2n}=0.$

As an example, $hS_6=(1,1,2,5,2,4,6,3,5,4,3,0,6)$ is a hooked Skolem sequence of order $6$ or, equivalently, the collection $\{(1,2),(3,5),(8,11),
(6,10), (4,9),(7,13)\}$.

A {\em (hooked) Langford sequence of length $n$ and defect $d$, $n>d$} is a sequence $L_d^n=(l_{i})$ of $2n$ $(2n+1)$ integers which satisfies:
\begin{enumerate}
\item for every $k\in \{d,d+1,\ldots,d+n-1\}$, there exist exactly two elements $l_{i},l_{j}\in L$ such that $l_{i}=l_{j}=k$,
\item if $l_{i}=l_{j}=k$ with $i<j$, then $j-i=k$,
\item in a hooked sequence $l_{2n}=0$.
\end{enumerate}

We noticed that Kramer's construction \cite{kramer} can be obtained using the canonical starter
 $v-2,v-4,\ldots,3,1,1,3,\ldots,v-4,v-2$ and taking the base blocks $\{0,i,b_i\}(\textup{mod}\;v)|i=1,2,\ldots,\frac{1}{2}(v-1)$. So, Kramer's construction can
 be obtained using Skolem-type sequences.

We prove next, that Kramer's construction for indecomposable triple systems produces simple designs.

\begin{theorem}\cite{kramer}{\label{kramer}}
The blocks $\{\{0,\alpha,-\alpha\}(\textup{mod}\;v)|\alpha=1,\ldots,\frac{1}{2}(v-1)\}$ for $v\equiv 1\;\textup{or}\;5\,(\textup{mod}\;6)$ form a cyclic, simple, and indecomposable 
three-fold triple system of order $v$.
\end{theorem}

\begin{proof}
Let $v=6n+1$. The design is cyclic and indecomposable \cite{kramer}. We prove that the cyclic three-fold triple systems produced by
 $\{\{0,\alpha,-\alpha\}(\textup{mod}\;v)|\alpha=1,\ldots,\frac{1}{2}(v-1)\}$  is also simple.

Suppose that the construction above produces $\{x,y,z\}$ as a repeated block. Any block $\{x,y,z\}$ is of the form $\{0,i,6n+1-i\}+k$ for some
 $i=1,2,\ldots,\frac{1}{2}(v-1)$ and $k \in \mathbb {Z}_{6n+1}$. Hence, if $\{x,y,z\}$ is a repeated block we have
$$\{0,i_1,6n+1-i_1\}+k_1=\{0,i_2,6n+1-i_2\}+k_2$$
whence,
$$\{0,i_2,6n+1-i_2\}=\{0,i_1,6n+1-i_1\}+k$$
for some $i_1,i_2\in\{1,2,\ldots,\frac{1}{2}(v-1)\}$ and some $k\in \mathbb{Z}_{6n+1}$.

If $k=0$, we have $i_2=6n+1-i_1$ and $i_1=6n+1-i_2$, which is impossible since $6n+1-i_1> i_2$ and $6n+1-i_2> i_1$ by definition (i.e., $i_1,i_2\in\{1,2,\ldots,3n\}$
 while $6n+1-i_1,\,6n+1-i_2\in\{3n+1,\ldots,6n\}$.

If $k=i_2$, we have $\begin{cases} i_1+i_2=6n+1 \\ 6n+1-i_1+i_2=6n+1-i_2 \end{cases}$ or 

$\begin{cases} i_1+i_2=6n+1-i_2 \\ 6n+1-i_1+i_2=6n+1. \end{cases}$

Since both $i_1$ and $i_2$ are at most $3n$, it is impossible to have $i_1+i_2=6n+1$. Also $i_1\neq i_2$.

If $k=6n+1-i_2$ we have $\begin{cases} i_1+6n+1-i_2=6n+1 \\ 6n+1-i_1+6n+1-i_2=i_2+6n+1 \end{cases}  \Leftrightarrow  \linebreak \begin{cases} i_1+i_2=6n+1-i_1 \\ 6n+1-i_2+i_1=6n+1 \end{cases}$ or $\begin{cases} i_1+6n+1-i_2=i_2 \\ 6n+1-i_1+6n+1-i_2=6n+1. \end{cases}$

Since $6n+1-i_2 > i_2$ , it is impossible to have $i_1+6n+1-i_2=i_2$.

It follows that our design is simple. The case for $v=6n+5$ is similar. \end{proof}

In order to completely solve the case $\lambda=3$, we have new constructions that give cyclic, simple, and indecomposable three-fold triple systems for $v\equiv 3~(\textup{mod}\;6)$, $v\neq 9$ and
$v\neq 24c+57$, $c\geq 2$.

\section{Simple Three-fold Cyclic Triple Systems}

\begin{lemma} \label{lem1} For every $n\equiv 0 ~ \textup{or} ~1 ~(\textup{mod}\,4)$, $n\geq 8$, there is a Skolem sequence of order $n$ in
which $s_1=s_2=1$ and $s_{2n-2}=s_{2n}=2$.

\end{lemma}

\begin{proof}

To get a Skolem sequence of order $n$ for $n\equiv 0 ~ \textup{or} ~1 ~(\textup{mod}\,4)$, $n\geq 8$, take $(1,1,hL_3^{n-2})$, replace the hook with 
a $2$ and add the other $2$ at the end of the sequence.

For $n=8$, take $hL_3^6=(8,3,5,7,3,4,6,5,8,4,7,0,6)$, for $n=12$ take $hL_3^{10}=(9,11,3,12,\linebreak 4, 3,7, 10,4,9,8,5,11,7,6,12,5,10,8,0,6)$ and for the
 remaining $hL_3^{n-2}$ hook a $hL_4^{n-3}$ (see \cite{simpson}, Theorem 2, Case 1) to $(3,0,0,3)$. 

For $n\equiv 1~(\textup{mod}~4)$, $n\geq9$, take $hL_3^{n-2}$ (see \cite{simpson}, Theorem 2, Case 1).
\end{proof}

\begin{ex}
From the above lemma we have $S_8=(1,1,8,3,5, 7,3,4,6,5,8,4,7,2,6,2)$, $S_{12}=(1,1,9,11,3,12,4,3,7, 10,4,9,8,5,11,7,6, 12,5,10,8,2,6,
2)$ and $S_{16}=(1,1,9,6,4,14,\linebreak 15,11,4,6,13,9,16,7,12,10,8, 5,11,14,7,15,5,13,8,10,12,3,16,2,3,2)$.
\end{ex}

We use the following construction to get cyclic TS$_3(2n+1)$ for $n\equiv 0$ or $1$ (mod $4$):

\begin{construction} \cite{silvesan} {\label {c5}}
Let $S_n=(s_1,s_2,\ldots,s_{2n})$ be a Skolem sequence of order $n$ and let $\{(a_i,b_i)|1\leq i\leq n\}$ be the pairs of positions
 in $S_n$ for which $b_i-a_i=i$. Then the set $\cal{F}$=$\{\{0,i,b_i\}|1\leq i\leq n\}(\textup{mod}\; 2n+1)$ is a $(2n+1,3,3)-DF$. Hence, the set of triples in $\cal{F}$ form the base blocks of a cyclic TS$_3(2n+1)$.
\end{construction}

Then, we apply Construction \ref{c5} to the Skolem sequences given by Lemma \ref{lem1} to get cyclic three-fold triple systems that are simple and indecomposable.

\begin{construction}{\label{c7}}
Let $S_n=(s_1,s_2,\ldots,s_{2n})$ be a Skolem sequence of order $n$ given by Lemma \ref{lem1}, and let $\{(a_i,b_i)|1\leq i\leq n\}$ be the pairs of positions
 in $S_n$ for which $b_i-a_i=i$. Then the set $\cal{F}$=$\{\{0,i,b_i\}|1\leq i\leq n\}(\textup{mod}\; 2n+1)$ form the base blocks of a cyclic, simple, and 
 indecomposable TS$_3(2n+1)$.
\end{construction}

\begin{ex}
If we apply Construction \ref{c7} to the Skolem sequence of order $8$: $(1,1,8,3,5,7,\linebreak 3,4,6,5,8,4,7,2,6,2)$ we get the base blocks 
$\{\{0,1,2\}, \{0,2,16\},
\{0,3,7\},\{0,4,12\},\{0,5,10\},\linebreak \{0,6,15\},\{0,7,13\},\{0,8,11\}\} (\textup{mod}~17)$. These base blocks form a cyclic TS$_3(17)$ by Construction \ref{c5}.
We are going to prove next that this design is also indecomposable and simple.
\end{ex}

\begin{theorem} \label{th1}
The TS$_3(6n+3)$, $n\geq2$ produced by applying Construction \ref{c7} are simple, except for $v=24c+57$, $c\geq 2$.
\end{theorem}

\begin{proof}

Let $v=2n+1$, $n\equiv 0 ~ \textup{or} ~1\,(\textup{mod}\,4)$, $n\geq 8$. 

Suppose that the construction above produces $\{x,y,z\}$ as a repeated block. With regards to Construction \ref{c7}, any block $\{x,y,z\}$ is of the
form $\{0,i,b_i\}+k$ for some $i=1,2,\ldots,n$ and $k \in \mathbb {Z}_{2n+1}$. Hence, if $\{x,y,z\}$ is a repeated block we have
$$\{0,i_1,b_{i_1}\}+k_1=\{0,i_2,b_{i_2}\}+k_2$$
whence,
$$\{0,i_2,b_{i_2}\}=\{0,i_1,b_{i_1}\}+k$$
for some $i_1,i_2\in\{1,2,\ldots,n\}$ and some $k\in \mathbb{Z}_{2n+1}$.

If $k=0$, we have $i_2=b_{i_1}$ and $i_1=b_{i_2}$ which is impossible since $b_{i_1}\geq i_1+1$ and $b_{i_2}\geq i_2+1$ from the definition of a Skolem sequence.

If $k=i_2$, we have $\begin{cases} i_1+i_2=2n+1 \\ b_{i_1}+i_2=b_{i_2} \end{cases}$ or $\begin{cases} i_1+i_2=b_{i_2} \\ b_{i_1}+i_2=2n+1. \end{cases}$

Since both $i_1$ and $i_2$ are at most $n$, it is impossible to have $i_1+i_2=2n+1$.

If $k=b_{i_2}$, we have $\begin{cases} i_1+b_{i_2}=2n+1 \\ b_{i_1}+b_{i_2}=i_2+2n+1 \end{cases} \Leftrightarrow \begin{cases} i_1+i_2=b_{i_1} \\ b_{i_2}+i_1=2n+1 \end{cases}$ or $\begin{cases} i_1+b_{i_2}=i_2 \\ b_{i_1}+b_{i_2}=2n+1. \end{cases}$

Since $b_{i_2} > i_2$ , it is impossible to have $i_1+b_{i_2}=i_2$.

So, to prove that a system has no repeated blocks is enough to show that $\begin{cases} i_1+i_2=b_{i_2} 
\\ b_{i_1}+i_2=2n+1 \end{cases}$ or $\begin{cases} i_1+i_2=b_{i_1} \\ b_{i_2}+i_1=2n+1 \end{cases}$ are not satisfied. 
Also, we show that $i=\frac{v}{3}$ and $b_{i}=\frac{2v}{3}$ is not allowed, which means that our systems has no short orbits.

For $n=8$ and $n=12$, it is easy to see that the Skolem sequences of order $n$ given by Lemma \ref{lem1} produce simple designs.

For $n\equiv 0(\textup{mod}~4),~n\geq 16$, let $S_n$ be the Skolem sequence given by Lemma \ref{lem1}. This Skolem sequence is constructed using 
the hooked Langford sequence $hL_4^{n-3}$ from \cite{simpson}, Theorem 2, Case 1. Since $d=4$, will use only lines $(1)-(7),(14),(8*),(10*)$ and 
$(11*)$ in Simpson's Table. Note that $n-3=9+4r$ in Simpson's Table, so $n=12+4r$ and $v=25+8r$, $r\geq 1$ in this case. Because we add the pair
$(1,1)$ at the beginning of the Langford sequence $hL_4^{n-3}$, $a_i$ and $b_i$ will be shifted to the right by two positions. To make it easier 
for the reader, we give in Table \ref{hL4n-3} the $hL_4^{n-3}$ taken from Simpson's Table and adapted for our case.

\begin{table}[!h]
\begin{center}
\begin{tabular}{ccccc}
\hline
& $a_i+2$ & $b_i+2$ & $i=b_i-a_i$ & $0\leq j\leq $\\
\hline
$(1)$ & $2r+3-j$ & $2r+7+j$ & $4+2j$ & $r$\\
$(2)$ & $r+2-j$ & $3r+9+j$ & $2r+7+2j$ & $r-1$\\
$(3)$ & $6r+12-j$ & $6r+17+j$ & $5+2j$ & $r-1$\\
$(4)$ & $5r+12-j$ & $7r+18+j$ & $2r+6+2j$ &$r$\\
$(5)$ & $3r+8$ & $7r+17$ & $4r+9$ &-\\
$(6)$ & $4r+9$ & $8r+21$ & $4r+12$ &-\\
$(7)$ & $2r+6$ & $6r+13$ & $4r+7$ & -\\
$(14)$ & $2r+5$ & $6r+16$ & $4r+11$ & -\\
$(8*)$ & $4r+11$ & $8r+19$ & $4r+8$ & -\\
$(10*)$ & $4r+10$ & $6r+15$ & $2r+5$ & -\\
$(11*)$ & $2r+4$ & $6r+14$ & $4r+10$ & -\\
\hline

\end{tabular}
\caption{$hL_4^{n-3}$} \label{hL4n-3}
\end{center}
\end{table}
So, the base blocks of the cyclic designs 
produced by Construction \ref{c7} are $\{0,1,2\},\{0,2,v-1\},\{0,3,v-2\}$ and $\{0,i,b_i+2\}$ for $i=4,\ldots,n$ and $i=b_i-a_i$.

We show first that $i=\frac{v}{3}$ and $b_{i}+2=\frac{2v}{3}$ is not allowed in the above system. In the first three base blocks is obvious
 that $i\neq \frac{v}{3}$. For the remaining base blocks we check lines $(1)-(7),(14),(8*),(10*)$ and $(11*)$ in Table \ref{hL4n-3}.

Line $(1)$ $\begin{cases} 4+2j=\frac{25+8r}{3} \\ 2r+7+j=\frac{2(25+8r)}{3} \end{cases} \Leftrightarrow \emptyset$. 

Line $(2)$ $\begin{cases} 2r+7+2j=\frac{25+8r}{3} \\ 3r+9+j=\frac{2(25+8r)}{3} \end{cases} \Leftrightarrow \emptyset$.

Line $(3)$ $\begin{cases} 5+2j=\frac{25+8r}{3} \\ 6r+17+j=\frac{2(25+8r)}{3} \end{cases} \Leftrightarrow \emptyset$. 

Line $(4)$ $\begin{cases} 2r+6+2j=\frac{25+8r}{3} \\ 7r+18+j=\frac{2(25+8r)}{3} \end{cases} \Leftrightarrow \emptyset$.

Line $(5)$ $\begin{cases} 4r+9=\frac{25+8r}{3} \\ 7r+17=\frac{2(25+8r)}{3} \end{cases} \Leftrightarrow \emptyset$. 

Line $(6)$ $\begin{cases} 4r+12=\frac{25+8r}{3} \\ 8r+21=\frac{2(25+8r)}{3} \end{cases} \Leftrightarrow \emptyset$.

Line $(7)$ $\begin{cases} 4r+7=\frac{25+8r}{3} \\ 6r+13=\frac{2(25+8r)}{3} \end{cases} \Leftrightarrow \emptyset$. 

Line $(14)$ $\begin{cases} 4r+11=\frac{25+8r}{3} \\ 6r+16=\frac{2(25+8r)}{3} \end{cases} \Leftrightarrow \emptyset$.

Line $(8*)$ $\begin{cases} 4r+8=\frac{25+8r}{3} \\ 8r+19=\frac{2(25+8r)}{3} \end{cases} \Leftrightarrow \emptyset$. 

Line $(10*)$ $\begin{cases} 2r+5=\frac{25+8r}{3} \\ 6r+15=\frac{2(25+8r)}{3} \end{cases} \Leftrightarrow \emptyset$.

Line $(11*)$ $\begin{cases} 4r+10=\frac{25+8r}{3} \\ 6r+14=\frac{2(25+8r)}{3} \end{cases} \Leftrightarrow \emptyset$.

Therefore, this systems has no short orbits.

Next, we have to check that $\begin{cases} i_1+i_2=b_{i_2} 
\\ b_{i_1}+i_2=2n+1 \end{cases}$ or $\begin{cases} i_1+i_2=b_{i_1} \\ b_{i_2}+i_1=2n+1 \end{cases}$ are not satisfied.

Lines $(1)-(1)$: $\begin{cases} 4+2j_1+4+2j_2=2r+7+j_2 \\ 4+2j_2+2r+7+j_1=25+8r \end{cases} \Leftrightarrow \begin{cases} j_1=\frac{-2r-15}{3} \\ 
j_2=\frac{10r+27}{3} \end{cases}$ which is impossible since $j_1\geq 0$ and also integer.

Lines $(1)-(2)$: $\begin{cases} 4+2j_1+2r+7+2j_2=3r+9+j_2 \\ 2r+7+2j_2+2r+7+j_1=25+8r \end{cases} \Leftrightarrow \begin{cases} j_1=\frac{-2r-15}{3} \\ j_2=r-2-2j_1
\end{cases}$
which is impossible since $j_1\geq 0$ and also integer.

Lines $(1)-(3)$: $\begin{cases} 4+2j_1+5+2j_2=6r+17+j_2 \\ 5+2j_2+2r+7+j_1=25+8r \end{cases} \Leftrightarrow \begin{cases} j_1=j_2-5 \\ j_2=2r+9\end{cases}$
which is impossible since $j_2\leq r-1$.

Lines $(1)-(4)$: $\begin{cases} 4+2j_1+2r+6+2j_2=7r+18+j_2 \\ 2r+6+2j_2+2r+7+j_1=25+8r \end{cases} \Leftrightarrow \begin{cases} j_1=j_2+r-4 \\ j_2=\frac{3r+16}{3} \end{cases}$
which is impossible since $j_2\leq r$.

Lines $(1)-(5)$: $\begin{cases} 4r+2j+13=7r+17 \\ j+6r+16=25+8r \end{cases} \Leftrightarrow \emptyset$. 

Lines $(1)-(6)$: $\begin{cases} 4r+2j+16=8r+21 \\ j+6r+19=25+8r \end{cases} \Leftrightarrow \emptyset$.

Lines $(1)-(7)$: $\begin{cases} 4r+2j+11=6r+13 \\ j+6r+14=25+8r \end{cases} \Leftrightarrow \emptyset$. 

Lines $(1)-(14)$: $\begin{cases} 4r+2j+15=6r+16 \\ j+6r+18=25+8r \end{cases} \Leftrightarrow \emptyset$.

Lines $(1)-(8*)$: $\begin{cases} 4r+2j+12=8r+19 \\ j+6r+15=25+8r \end{cases} \Leftrightarrow \emptyset$. 

Lines $(1)-(10*)$: $\begin{cases} 2r+2j+9=6r+15 \\ j+4r+12=25+8r \end{cases} \Leftrightarrow \emptyset$.

Lines $(1)-(11*)$: $\begin{cases} 4r+2j+14=6r+14 \\j+6r+17=25+8r \end{cases} \Leftrightarrow \emptyset$.

We implemented a program in Mathematica that checks all the pairs of rows in Simpson's table using the above approach. The code for the program and the results can be found in Appendix. 
From the results, we can easily see that if we check any combination of two lines in Simpson's Table the conditions are not satisfied in almost all of the cases. There are 
two cases where this conditions are satified. The first case is when we check line 3 with line 1, and we get that for $r=4+3c$, $j_1=2c$, and $j_2=6+2c$, $c\geq 2$ the system is
not simple. This implies that our system is not simple when $v=24c+57$, $c\geq2$. The second case is when we check line 3 with line 2. Here, we get $r=5$ and therefore $v=59$.
But $v=59$ is not congruent to $3$ (mod $6$). A TS$_3(59)$ is simple, cyclic, and indecomposable by Theorem \ref{kramer}.

For $n\equiv 1(\textup{mod}~4), n\geq 9$, let $S_n$ be the Skolem sequence given by Lemma \ref{lem1}. This Skolem sequence is constructed
 using a $hL_3^{n-2}$ (\cite{simpson}, Theorem 2, Case 1). Since $d=3$, will use only lines
 $(1)-(6),(14),(7'),(8')$ and $(10')$ in Simpson's Table. Note that $n-2=7+4r$ in Simpson's Table, so $n=9+4r$ and $v=19+8r$ in this case.
 Because we add the pair $(1,1)$ at the beginning of the Langford sequence $hL_3^{n-2}$, $a_i$ and $b_i$ will be shifted to the right
 by two positions. 
 
Table \ref{hL3n-2} gives the $hL_3^{n-2}$ from Simpson's Table adapted to our case.

\begin{table}[!h]
\begin{center}
\begin{tabular}{ccccc}
\hline
& $a_i+2$ & $b_i+2$ & $i=b_i-a_i$ & $0\leq j\leq $\\
\hline
$(1)$ & $2r+3-j$ & $2r+6+j$ & $3+2j$ & $r$\\
$(2)$ & $r+2-j$ & $3r+8+j$ & $2r+6+2j$ & $r-1$\\
$(3)$ & $6r+10-j$ & $6r+14+j$ & $4+2j$ & $r-1$\\
$(4)$ & $5r+10-j$ & $7r+15+j$ & $2r+5+2j$ &$r$\\
$(5)$ & $3r+7$ & $7r+14$ & $4r+7$ &-\\
$(6)$ & $4r+8$ & $8r+17$ & $4r+9$ &-\\
$(14)$ & $2r+4$ & $6r+12$ & $4r+8$ & -\\
$(7')$ & $2r+5$ & $6r+11$ & $4r+6$ & -\\
$(10')$ & $4r+9$ & $6r+13$ & $2r+4$ & -\\
\hline

\end{tabular}
\caption{$hL_3^{n-2}$} \label{hL3n-2}
\end{center}
\end{table}

So, the base blocks of the cyclic designs produce by Construction \ref{c7} are $\{0,1,2\},\{0,2,v-1\}$ and $\{0,i,b_i+2\}$ for 
$i=3,\ldots,n$. Using the same argument as before, we show that these designs are simple.

First, we show that $i=\frac{v}{3}$ and $b_{i}+2=\frac{2v}{3}$ is not allowed in the above system. In the first two base blocks is obvious
 that $i\neq \frac{v}{3}$. For the remaining base blocks we check lines $(1)-(6),(14),(7')$ and $(10')$ in Table \ref{hL3n-2}.

Line $(1)$ $\begin{cases} 3+2j=\frac{19+8r}{3} \\ 2r+6+j=\frac{2(19+8r)}{3} \end{cases} \Leftrightarrow \emptyset$. 

Line $(2)$ $\begin{cases} 2r+6+2j=\frac{19+8r}{3} \\ 3r+8+j=\frac{2(19+8r)}{3} \end{cases} \Leftrightarrow \emptyset$.

Line $(3)$ $\begin{cases} 4+2j=\frac{19+8r}{3} \\ 6r+14+j=\frac{2(19+8r)}{3} \end{cases} \Leftrightarrow \emptyset$. 

Line $(4)$ $\begin{cases} 2r+5+2j=\frac{19+8r}{3} \\ 7r+15+j=\frac{2(19+8r)}{3} \end{cases} \Leftrightarrow \emptyset$.

Line $(5)$ $\begin{cases} 4r+7=\frac{19+8r}{3} \\ 7r+14=\frac{2(19+8r)}{3} \end{cases} \Leftrightarrow \emptyset$.

Line $(6)$ $\begin{cases} 4r+9=\frac{19+8r}{3} \\ 8r+17=2\frac{19+8r}{3} \end{cases} \Leftrightarrow \emptyset$.

Line $(14)$ $\begin{cases} 4r+8=\frac{19+8r}{3} \\ 6r+12=\frac{2(19+8r)}{3} \end{cases} \Leftrightarrow \emptyset$. 

Line $(7')$ $\begin{cases} 4r+6=\frac{19+8r}{3} \\ 6r+11=\frac{2(19+8r)}{3} \end{cases} \Leftrightarrow \emptyset$.

Line $(10')$ $\begin{cases} 2r+4=\frac{19+8r}{3} \\ 6r+13=\frac{2(19+8r)}{3} \end{cases} \Leftrightarrow \emptyset$.

Next, we have to check that $\begin{cases} i_1+i_2=b_{i_2} 
\\ b_{i_1}+i_2=2n+1 \end{cases}$ or $\begin{cases} i_1+i_2=b_{i_1} \\ b_{i_2}+i_1=2n+1 \end{cases}$ are not satisfied. As with the previous case, the results can be found in Appendix. As before, when we check line 3 and line 1 the conditions are satisfied. But, in this case $v=24c+35$, $c\geq1$ which is not congruent to $3$ (mod $6$).
So, a TS$_3(24c+35)$ for $c\geq 1$ is cyclic, simple, and indecomposable by Theorem \ref{kramer}.\end{proof}

\begin{lemma} \label{lem2} For every $n\equiv 2 ~ \textup{or} ~3 ~(\textup{mod}\,4)$, $n\geq 7$, there is a hooked Skolem sequence of order $n$ in which $s_1=s_2=1$ and
 $s_{2n-1}=s_{2n+1}=2$. 
\end{lemma}

\begin{proof}

For $n\equiv 2 ~ \textup{or} ~3 ~(\textup{mod}\,4)$, $n\geq 7$, take hS$_n=(1,1,L_3^{n-2},2,0,2)$.

When $n\equiv 2~(\textup{mod}~4)$, take $L_3^{n-2}$ (\cite{simpson}, Theorem 1, Case 3). When $n\equiv 3~(\textup{mod}~4)$, take 
$L_3^5=(6,7,3,4,5,3,6,4,7,5) $ and for $n\geq 11$ take  $L_3^{n-2}$ (see \cite{bermond}, Theorem 2).
\end{proof}

We are going to use the following construction to get cyclic TS$_3(2n+1)$ for $n\equiv 2$ or $3$ (mod $4$):

\begin{construction} \cite{silvesan} {\label{c6}}
Let $hS_n=(s_1,s_2,\ldots,s_{2n-1},s_{2n+1})$ be a hooked Skolem sequence of order $n$ and let $\{(a_i,b_i)|1\leq i\leq n\}$ be
 the pairs of positions in $hS_n$ for which $b_i-a_i=i$. Then the set $\cal{F}$=$\{\{0,i,b_i+1\}|1\leq i\leq n\}(\textup{mod}\;2n+1)$ is a $(2n+1,3,3)-DF$. Hence, the set of triples in $\cal{F}$ form the base blocks of a cyclic TS$_3(2n+1)$.
\end{construction}

Then, we apply Construction \ref{c6} to the hooked Skolem sequences given by Lemma \ref{lem2} to get cyclic TS$_3(2n+1)$ for $n\equiv 2$ or $3$ (mod $4$) that are simple and indecomposable.

\begin{construction} {\label{c8}}
Let $hS_n=(s_1,s_2,\ldots,s_{2n-1},s_{2n+1})$ be a hooked Skolem sequence of order $n$ given by Lemma \ref{lem2}, and let $\{(a_i,b_i)|1\leq i\leq n\}$ be
 the pairs of positions in $hS_n$ for which $b_i-a_i=i$. Then, the set $\cal{F}$=$\{\{0,i,b_i+1\}|1\leq i\leq n\}(\textup{mod}\;2n+1)$ form the base 
 blocks of a cyclic, simple, and indecomposable TS$_3(2n+1)$.
\end{construction}

\begin{ex}
If we apply Construction \ref{c8} to the hooked Skolem sequence of order $7$: $(1,1,6,7,3,4,5,3,6,4,7,5,2,0,2)$ we get the base blocks $\{\{0,1,3\},\linebreak \{0,2,1\},
\{0,3,9\},\{0,4,11\}, \{0,5,13\},\{0,6,10\},\{0,7,12\}\} (\textup{mod}~15)$. These base blocks form a cyclic TS$_3(15)$ by Construction 
\ref{c6}. We are going to prove next, that this design is indecomposable and simple.
\end{ex}

\begin{theorem} \label{th2}
The TS$_3(6n+3)$, $n\geq 2$, produced by applying Construction \ref{c8}, are simple.
\end{theorem}
 
\begin{proof}
The proof is similar to Theorem \ref{th1}. Let $v=2n+1$, $n\equiv 2 ~ \textup{or} ~3\,(\textup{mod}\,4)$, $n\geq 10$. 

For $n\equiv 2(\textup{mod}~4),~n\geq 10$, let $hS_n$ be the hooked Skolem sequence given by Lemma \ref{lem2}. This hooked Skolem sequence is 
constructed using the Langford sequence $L_3^{n-2}$ from \cite{simpson}, Theorem 1, Case 3. Since $d=3$, will use only lines $(1)-(4),(6),(9),(11)$ and 
$(13)$ in Simpson's Table. Note that $m=n-2=4r$ in Simpson's Table, so $n=4r+2$, $v=8r+5$, $r\geq 2$, $d=3$, $s=1$, in this case. Because we add the pair $(1,1)$ at the
beginning of the hooked Langford sequence $hL_3^{n-2}$, $a_i$ and $b_i$ will be shifted to the right by two positions. To make it easier for the reader, we 
give in Table \ref{L3n-2}, the $L_3^{n-2}$ taken from Simpson's Table and adapted for our case (omit row $(1)$ when $r=2$).

\begin{table}[!h]
\begin{center}
\begin{tabular}{ccccc}
\hline
& $a_i+2$ & $b_i+2$ & $i=b_i-a_i$ & $0\leq j\leq $\\
\hline
$(1)$ & $2r-j$ & $2r+4+j$ & $4+2j$ & $r-3$\\
$(2)$ & $r+2-j$ & $3r+3+j$ & $2r+1+2j$ & $r-1$\\
$(3)$ & $6r+1-j$ & $6r+4+j$ & $3+2j$ & $r-2$\\
$(4)$ & $5r+2-j$ & $7r+4+j$ & $2r+2+2j$ &$r-2$\\
$(6)$ & $2r+3$ & $4r+3$ & $2r$ &-\\
$(9)$ & $3r+2$ & $7r+3$ & $4r+1$ &-\\
$(11)$ & $2r+1$ & $6r+3$ & $4r+2$ & -\\
$(13)$ & $2r+2$ & $6r+2$ & $4r$ & -\\
\hline
\end{tabular}
\caption{$L_3^{n-2}$} \label{L3n-2}
\end{center}
\end{table}
So, the base blocks of the cyclic designs 
produce by Construction \ref{c8} are $\{0,1,3\},\{0,2,1\}$ and $\{0,i,b_i+2+1\}$ for $i=3,\ldots,n$ and $i=b_i-a_i$.

First, we show that $i=\frac{v}{3}$ and $b_{i}+2+1=\frac{2v}{3}$ is not allowed in the above system. In the first two base blocks is obvious
 that $i\neq \frac{v}{3}$. For the remaining base blocks we check lines $(1)-(4),(6),(9),(11)$ and $(13)$ in Table \ref{L3n-2}.

Line $(1)$ $\begin{cases} 4+2j=\frac{8r+5}{3} \\ 2r+5+j=\frac{2(8r+5)}{3} \end{cases} \Leftrightarrow r=\frac{1}{4}$ which is impossible since 
$r\geq 2$ and also integer.

Line $(2)$ $\begin{cases} 2r+1+2j=\frac{8r+5}{3} \\ 3r+4+j=\frac{2(8r+5)}{3} \end{cases} \Leftrightarrow r=\frac{1}{2}$ which is impossible since 
$r\geq 2$ and also integer.

Line $(3)$ $\begin{cases} 3+2j=\frac{8r+5}{3} \\ 6r+5+j=\frac{2(8r+5)}{3} \end{cases} \Leftrightarrow r=-\frac{1}{2}$ which is impossible since
$r\geq 2$ and also integer.

Line $(4)$ $\begin{cases} 2r+2+2j=\frac{8r+5}{3} \\ 7r+5+j=\frac{2(8r+5)}{3} \end{cases} \Leftrightarrow r=-\frac{3}{4}$ which is impossible since
$r\geq 2$ and also integer.

Line $(6)$ $\begin{cases} 2r=\frac{8r+5}{3} \\ 4r+4=\frac{2(8r+5)}{3} \end{cases} \Leftrightarrow \emptyset$. Line $(9)$ $\begin{cases} 4r+1=\frac{8r+5}{3} \\ 7r+4=2\frac{8r+5}{3} \end{cases} \Leftrightarrow \emptyset$.

Line $(11)$ $\begin{cases} 4r+2=\frac{8r+5}{3} \\ 6r+4=\frac{2(8r+5)}{3} \end{cases} \Leftrightarrow \emptyset$. Line $(13)$ $\begin{cases} 4r=\frac{8r+5}{3} \\ 6r+3=\frac{2(8r+5)}{3} \end{cases} \Leftrightarrow \emptyset$.

Next, we have to show that $\begin{cases} i_1+i_2=b_{i_2} 
\\ b_{i_1}+i_2=2n+1 \end{cases}$ or $\begin{cases} i_1+i_2=b_{i_1} \\ b_{i_2}+i_1=2n+1 \end{cases}$ are not satisfied. The results for this can be found Appendix. As before, when we check line $3$ with line $1$, the conditions are satisfied. But $v=24c+5$, $c\geq2$ in this case which is not congruent to $3$ (mod $6$). So, 
by Theorem \ref{kramer}, there exists a cyclic, simple, and indecomposable TS$_3(24c+5)$ for $c\geq 2$.

For $n\equiv 3(\textup{mod}~4), n\geq 11$, let $hS_n$ be the hooked Skolem sequence given by Lemma \ref{lem2}. This hooked Skolem sequence is constructed
 using a $L_3^{n-2}$ (\cite{bermond}, Theorem 2). Since $d=3$ will use only lines
 $(1)-(4),(6)-(10)$ in \cite{bermond}. Note that $m=n-2=4r+1, r\geq 2$, $e=4$ in \cite{bermond}, so $n=4r+3$ and $v=8r+7$ in this case.
 Because we add the pair $(1,1)$ at the beginning of the Langford sequence $L_3^{n-2}$, $a_i$ and $b_i$ will be shifted to the right
 by two positions. 
 
Table \ref{L3n-2germa} gives the $L_3^{n-2}$ from \cite{bermond} adapted to our case.

\begin{table}[!h]
\begin{center}
\begin{tabular}{ccccc}
\hline
& $a_i+2$ & $b_i+2$ & $i=b_i-a_i$ & $0\leq j\leq $\\
\hline
$(1)$ & $2r+2-j$ & $2r+6+j$ & $4+2j$ & $r-2$\\
$(2)$ & $r+2-j$ & $3r+5+j$ & $2r+3+2j$ & $r-2$\\
$(3)$ & $3$ & $4r+4$ & $4r+1$ & -\\
$(4)$ & $2r+4$ & $4r+5$ & $2r+1$ &-\\
$(6)$ & $r+3$ & $5r+5$ & $4r+2$ &-\\
$(7)$ & $2r+5$ & $6r+5$ & $4r$ &-\\
$(8$ & $2r+3$ & $6r+6$ & $4r+3$ & -\\
$(9)$ & $6r+4-j$ & $6r+7+j$ & $3+2j$ & $r-2$\\
$(10)$ & $5r+4-j$ & $7r+6+j$ & $2r+2+2j$ & $r-2$\\
\hline

\end{tabular}
\caption{$L_3^{n-2}$} \label{L3n-2germa}
\end{center}
\end{table}

So, the base blocks of the cyclic designs produce by Construction \ref{c8} are $\{0,1,3\},\{0,2,1\}$ and $\{0,i,b_i+2+1\}$ for 
$i=3,\ldots,n$. Using the same argument as before, we show that these designs are simple.

First we show that $i=\frac{v}{3}$ and $b_{i}+2+1=\frac{2v}{3}$ is not allowed in the above system. In the first two base blocks is obvious
 that $i\neq \frac{v}{3}$. For the remaining base blocks we check lines $(1)-(4),(6)-(10)$ in Table \ref{L3n-2germa}.

Line $(1)$ $\begin{cases} 4+2j=\frac{8r+7}{3} \\ 2r+7+j=\frac{2(8r+7)}{3} \end{cases} \Leftrightarrow r=\frac{3}{4}$ which is impossible since 
$r\geq 2$ and also integer.

Line $(2)$ $\begin{cases} 2r+3+2j=\frac{8r+7}{3} \\ 3r+6+j=\frac{2(8r+7)}{3} \end{cases} \Leftrightarrow r=\frac{1}{2}$ which is impossible since 
$r\geq 2$ and also integer.

Line $(3)$ $\begin{cases} 4r+1=\frac{8r+7}{3} \\ 4r+5=\frac{2(8r+7)}{3} \end{cases} \Leftrightarrow \emptyset$. Line $(4)$ $\begin{cases} 2r+1=\frac{8r+7}{3} \\ 4r+6=\frac{2(8r+7)}{3} \end{cases} \Leftrightarrow \emptyset$.

Line $(6)$ $\begin{cases} 4r+2=\frac{8r+7}{3} \\ 5r+6=\frac{2(8r+7)}{3} \end{cases} \Leftrightarrow \emptyset$. Line $(7)$ $\begin{cases} 4r=\frac{8r+7}{3} \\ 6r+6=\frac{2(8r+7)}{3} \end{cases} \Leftrightarrow \emptyset$.

Line $(8)$ $\begin{cases} 4r+3=\frac{8r+7}{3} \\ 6r+7=\frac{2(8r+7)}{3} \end{cases} \Leftrightarrow \emptyset$. Line $(9)$ $\begin{cases} 3+2j=\frac{8r+7}{3} \\ 6r+8+j=
\frac{2(8r+7)}{3} \end{cases} \Leftrightarrow \emptyset$.

Line $(10)$ $\begin{cases} 2r+2+2j=\frac{8r+7}{3} \\ 7r+7+j=\frac{2(8r+7)}{3} \end{cases} \Leftrightarrow \emptyset$.

Next, we have to check that $\begin{cases} i_1+i_2=b_{i_2} 
\\ b_{i_1}+i_2=2n+1 \end{cases}$ or $\begin{cases} i_1+i_2=b_{i_1} \\ b_{i_2}+i_1=2n+1 \end{cases}$ are not satisfied. The results for this can be found in Appendix. Here, for $v=3c-1$, $c\geq 4$ and for $v=55$ the conditions are satisfied but these orders are not congruent to $3$ (mod $6$). Therefore, by Theorem \ref{kramer},
 there exists cyclic, simple, and indecomposable TS$_3(3c-1)$ for $c\geq 4$ and cyclic, simple, and indecomposable TS$_3(55)$.
\end{proof}

\section{Indecomposable Three-fold  Cyclic Triple Systems}

In this section, we prove that the Constructions \ref{c7} and \ref{c8} produce indecomposable three-fold triple systems for $v\equiv~3~(\textup{mod}~6)$,
$v\geq 15$.

\begin{theorem}\label{ind1}
The TS$_3(v)$ produced by Constructions \ref{c7} and \ref{c8} are indecomposable for every $v\equiv~3~(\textup{mod}~6)$, $v\geq 15$.
\end{theorem}

\begin{proof}
Suppose that $v\equiv~3~(\textup{mod}~6)$ and write $v=2n+1$, $n\equiv 0 ~ \textup{or} ~1\,(\textup{mod}\,4)$, $n\geq 8$. 

Now, for an TS$_3(2n+1)$ to be decomposable, there must be a Steiner triple system STS($2n+1$) inside the TS$_3(2n+1)$. 

If $2n+1\equiv ~3~(\textup{mod}\,6)$, let $\{x_i,x_j,x_k\}$ be a triple using symbols from N$_{2n+1}=\{0,1,\ldots,2n\}$. Let
 $d_{ij}=\textup{min}~\{|x_i-x_j|,2n+1-|x_i-x_j|\}$ be the difference between $x_i$ and $x_j$. An STS($2n+1$) on N$_{2n+1}$ must
 have a set of triples with the property that each difference $d$, $1\leq d\leq n$, occurs exactly $2n+1$ times. Assume there is an
 STS($2n+1$) inside our TS$_3(2n+1)$ and let $f_{\alpha}$ be the number of triples inside the STS($2n+1$) which are a cyclic shift of $\{0,\alpha,b_{\alpha}\}$. 

It is enough to look at the first two base blocks of our TS$_3(2n+1)$. These are $\{0,1,2\}$ $(\textup{mod} ~ 2n+1)$ and $\{0,2,2n\}~(\textup{mod} ~ 2n+1)$.
Then the existence of an STS($2n+1$) inside our TS$_3(2n+1)$ requires that the equation $2f_1+f_2=2n+1$ must have a solution in nonnegative integers (we need
 the difference $1$ to occur exactly $2n+1$ times). 

{\bf Case \mbox {\boldmath $1$}: \mbox {\boldmath $f_1=1$}}

Suppose we choose one block from the orbit $\{0,1,2\}(\textup{mod}\;2n+1)$. Since this orbit uses the difference $1$ twice and the difference $2$, and the
 orbit $\{0,2,2n\}(\textup{mod} ~ 2n+1)$ uses the differences $1,\,2$ and $3$, whenever we pick one block from the first orbit we cannot choose three blocks from
 the second orbit (i.e., those blocks where the pairs $(0,1)$, $(0,2)$ and $(1,2)$ are included). So, we just have $2n-2$ blocks in the second orbit to choose from. 
But we need $2n-1$ blocks from the second orbit in order to cover difference $1$ exactly $2n+1$ times.

Therefore, we have no solution in this case.

{\bf Case \mbox {\boldmath$2$}: \mbox {\boldmath $f_1=2$}}

Since $f_2=\frac{2n-3}{2}$ is not an integer, we have no solution in this case.

{\bf Case \mbox {\boldmath$3$}: \mbox {\boldmath $f_1=3,5,\ldots, n\,(\textup{or}\,n-1)$}}

Similar to Case $1$. So, there is no solution in this case.

{\bf Case \mbox {\boldmath$4$}: \mbox {\boldmath $f_1=4,6,\ldots, n\,(\textup{or}\,n-1)$}}

Similar to Case $2$. So, there is no solution in this case.

{\bf Case \mbox {\boldmath$5$}: \mbox {\boldmath $f_1=0$}}

Note that our cyclic TS$_3(v)$ has no short orbits (Theorem \ref{th1}), while a cyclic STS$(v)$ will have a short orbit. Therefore, if a design 
exists inside our TS$_3(v)$, that design is not cyclic.

Now, we choose no block from the first orbit and all the blocks in the second orbit (i.e., $f_1=0,\;f_2=2n+1$). Therefore differences $1$, $2$ and $3$ are all
 covered each exactly $2n+1$ in the STS$(v)$. From the remaining $n-2$ orbits $\{0,i,b_i\},\,i\geq 3$ there will be two or three orbits which will use differences $2$ and $3$. 
Since differences $1$, $2$ and $3$ are already covered, we cannot choose any block from those orbits that uses these three differences. So, we are left with $n-4$
 or $n-5$ orbits to choose from.
We need to cover differences $4,5,\ldots,n$ ($n-3$ differences) each exactly $v=2n+1$ times.

We form a system of $n-3$ equations with $n-4$ or $n-5$ unknowns in the following way: when a difference appears in different orbits,
 the sum of the blocks that we choose from each orbit has to equal $v$, i.e., if difference $4$ appears in $\{0,5,b_5\}$, $\{0,7,b_7\}$ and $\{0,10,b_{10}\}$ we have
 $f_5+f_7+f_{10}=v$ or
 if difference $4$ appears in $\{0,7,b_7\}$ twice and in $\{0,9,b_9\}$ once, we have $2f_7+f_9=v$. The system that we form has two or three entries in each row non-zero
 while all the others entries will equal zero. The rows in the system can be rearranged so that we get an upper triangular matrix. Therefore, the system of equations is non-singular and it has the unique solution 
$f_{i_1}=f_{i_2}=\ldots =f_{i_k}=v$ for some $4\leq i_1,i_2,\ldots,i_k\leq n$ and $f_{j_1}=f_{j_2}=\ldots =f_{j_k}=0$ for some $4\leq j_1,j_2,\ldots,j_k\leq n$. But this
 implies that the STS$(v)$ inside our TS$_3(v)$ is cyclic, which is impossible.

Therefore, we have no solution in this case. It follows that our TS$_3(2n+1)$ is indecomposable.

Now, suppose that $v=2n+1$, $n\equiv 2 ~ \textup{or} ~3 ~(\textup{mod} ~4)$, $n\geq 7$. Let $f_{\alpha}$ be the number of triples inside the STS($2n+1$) which are a cyclic shift of $\{0,\alpha,b_{\alpha}+1\}$.
 Using the same argument as before, it is easy to show that the equation $2f_2+f_1=2n+1$ has no solution. Therefore our TS$_3(2n+1)$ is indecomposable.
\end{proof}

\section{Cyclic, Simple, and Indecomposable Three-fold Triple Systems}

\begin{theorem}
There exists cyclic, simple, and indecomposable three-fold triple systems, TS$_3(v)$, for every $v\equiv 1~(\textup{mod}~2)$, $v\geq 5, v\neq 9$ and $v\neq 24c+57$, $c\geq 2$.
\end{theorem}

\begin{proof}

Let $v\equiv 1~\textup{or}~5~(\textup{mod}~6)$ and take the base blocks $\{0,\alpha,-\alpha\}(\textup{mod}\;v)|\alpha=0,1,\ldots,\frac{1}{2}(v-1)$. By Theorem 
\ref{kramer}, these will be the base blocks of a cyclic, simple, and indecomposable three-fold triple system.

Let $v\equiv 3~(\textup{mod}~6)$, and write $v=2n+1$, $n\equiv 0 ~ \textup{or} ~1\,(\textup{mod}\,4)$, $n\geq 8$. Apply Construction \ref{c7} to the Skolem sequence 
of order $n$ given by Lemma \ref{lem1}. These designs are cyclic by Construction \ref{c5}, simple for all $v$ except $v=24c+57$, $c\geq2$ by Theorem \ref{th1} and indecomposable by Theorem \ref{ind1}.

Let $v\equiv 3~(\textup{mod}~6)$, and write $v=2n+1$, $n\equiv 2 ~ \textup{or} ~3\,(\textup{mod}\,4)$, $n\geq 7$. Apply Construction \ref{c8} to the hooked Skolem
 sequence of order $n$ given by Lemma \ref{lem2}. These designs are cyclic by Construction \ref{c6}, simple by Theorem \ref{th2} and indecomposable by
 Theorem \ref{ind1}.
\end{proof}

\section{Conclusion and Open Problems}

We constructed, using Skolem-type sequences, three-fold triple systems having all the properties of being cyclic, simple and indecomposable,
 for $v\equiv~3~(\textup{mod}~6)$ except for $v=9$ and $v=24c+57$, $c\geq2$. Our results, together with Kramer's results \cite{kramer}, completely solve the problem of finding 
three-fold triple systems having three properties: cyclic, simple, and indecomposable with some possible exceptions for $v=9$ and $v=24c+7$, $c\geq2$.

In our approach of finding cyclic, simple, and indecomposable three-fold triple systems, proving the simplicity of the designs was a 
tedious and long task. Another approach that we tried was in constructing three disjoint (i.e. no two pairs in the same 
positions) sequences of order $n$ and taking the
 base blocks $\{0,i,b_i+n\}, i=1,2,\ldots,n$. These base blocks form a cyclic TS$_3(6n+1)$. Also, someone can take three disjoint 
hooked sequences of order $n$ and taking the base blocks
 $\{0,i,b_i+n\}, i=1,2,\ldots,n$ together with the short orbit $\{0,\frac{v}{3},\frac{2v}{3}\}$. These base blocks form a cyclic 
TS$_3(6n+3)$.

\begin{ex}
For $n=5$, take the three disjoint hooked sequences of order $n$:
\begin{center}
$(1,1,4,1,1,0,4,2,3,2,0,3)$

$(2,3,2,3,3,0,3,4,1,1,0,4)$

$(4,5,5,5,4,0,5,5,5,2,0,2)$
\end{center}

Then the base blocks $\{0,1,7\},\{0,1,10\},\{0,1,15\},\{0,2,8\},\{0,2,15\},\linebreak \{0,2,17\},\{0,3,10\},
\{0,3,12\},\{0,3,17\},\{0,4,10\},\{0,4,17\},\{0,4,12\},\linebreak \{0,5,12\},\{0,5,13\},\{0,5,14\},\{0,11,22\}$ are the base blocks of a cyclic, simple, and indecomposable TS$_3(33)$.
\end{ex}

The simplicity is easy to prove here since the three hooked sequences that we used share no pairs in the same positions. On the other 
hand, on first inspection, to prove the indecomposability of such designs appears to be more difficult. Also, someone needs to find 
three disjoint such sequences for all admissible orders $n$.

\begin{problem}
Can the above approach of finding cyclic, simple, and indecomposable three-fold triple systems of order $v$ be generalized for all admissible orders?
\end{problem}

\begin{problem}
Are there cyclic, simple, and indecomposable TS$_3(24c+57)$, $c\geq 2$?
\end{problem}

\begin{problem}
Are there cyclic, simple, and indecomposable designs for $\lambda\geq 4$ and all admissible orders?
\end{problem}

\begin{problem}
For $\lambda\geq 3$ what is the spectrum of those $v$ for which there exists a cyclically indecomposable but decomposable cyclic TS$_\lambda(v)$?
\end{problem}


\begin{ex}

Let $V=\{0,1,2,3,4,5,6,7,8\}$. Then the blocks of a cyclic TS$_3(9)$ are:

$\{0,1,2\}, \{0,2,7\}, \{0,3,6\}, \{0,4,8\}$

$\{1,2,3\}, \{1,3,8\}, \{1,4,7\}, \{1,5,0\}$

$\{2,3,4\}, \{2,4,0\}, \{2,5,8\}, \{2,6,1\}$

$\{3,4,5\}, \{3,5,1\}, \{3,6,0\}, \{3,7,2\}$

$\{4,5,6\}, \{4,6,2\}, \{4,7,1\}, \{4,8,3\}$

$\{5,6,7\}, \{5,7,3\}, \{5,8,2\}, \{5,0,4\}$

$\{6,7,8\}, \{6,8,4\}, \{6,0,3\}, \{6,1,5\}$

$\{7,8,0\}, \{7,0,5\}, \{7,1,4\}, \{7,2,6\}$

$\{8,0,1\}, \{8,1,6\}, \{8,2,5\}, \{8,3,7\}$

This design is cyclic, simple, and is decomposable. The blocks $\{0,1,2\},\linebreak \{3,4,5\},\{6,7,8\}, \{1,3,8\},\{4,6,2\},\{7,0,5\},\{0,3,6\},\{1,4,7\},\{2,5,8\},\linebreak \{0,4,8\},
 \{3,7,2\},\{6,1,5\}$ form an STS$(9)$. The designs is cyclically indecomposable since no CSTS$(9)$ exists.
\end{ex}

\section{Appendix A: Mathematica Program}
\label{program}
\scriptsize
\begin{verbatim}
Ltable = {1, 2, 3, 4};
Ltable[[1]] = {{2*r + 3 - j, 2*r + 7 + j, 4 + 2*j, r},
   	{r + 2 - j, 3*r + 9 + j, 2*r + 7 + 2*j, r - 1},
   	{6*r + 12 - j, 6*r + 17 + j, 5 + 2*j, r - 1},
   	{5*r + 12 - j, 7*r + 18 + j, 2*r + 6 + 2*j, r},
   	{3*r + 8, 7*r + 17, 4*r + 9, 0},
   	{4*r + 9, 8*r + 21, 4*r + 12, 0},
   	{2*r + 6, 6*r + 13, 4*r + 7, 0},
   	{2*r + 5, 6*r + 16, 4*r + 11, 0},
   	{4*r + 11, 8*r + 19, 4*r + 8, 0},
   	{4*r + 10, 6*r + 15, 2*r + 5, 0},
   	{2*r + 4, 6*r + 14, 4*r + 10, 0}
   };
Ltable[[2]] = {{2*r + 3 - j, 2*r + 6 + j, 3 + 2*j, r},
   	{r + 2 - j, 3*r + 8 + j, 2*r + 6 + 2*j, r - 1},
   	{6*r + 10 - j, 6*r + 14 + j, 4 + 2*j, r - 1},
   	{5*r + 10 - j, 7*r + 15 + j, 2*r + 5 + 2*j, r},
   	{3*r + 7, 7*r + 14, 4*r + 7, 0},
   	{4*r + 8, 8*r + 17, 4*r + 9, 0},
   	{2*r + 4, 6*r + 12, 4*r + 8, 0},
   	{2*r + 5, 6*r + 11, 4*r + 6, 0},
   	{4*r + 9, 6*r + 13, 2*r + 4, 0}
   };
Ltable[[3]] = {{2*r - j, 2*r + 4 + j, 4 + 2*j, r - 3},
   	{r + 2 - j, 3*r + 3 + j, 2*r + 1 + 2*j, r - 1},
   	{6*r + 1 - j, 6*r + 4 + j, 3 + 2*j, r - 2},
   	{5*r + 2 - j, 7*r + 4 + j, 2*r + 2 + 2*j, r - 2},
   	{2*r + 3, 4*r + 3, 2*r, 0},
   	{3*r + 2, 7*r + 3, 4*r + 1, 0},
   	{2*r + 1, 6*r + 3, 4*r + 2, 0},
   	{2*r + 2, 6*r + 2, 4*r, 0}
   };
Ltable[[4]] = {{2*r + 2 - j, 2*r + 6 + j, 4 + 2*j, r - 2},
   	{r + 2 - j, 3*r + 5 + j, 2*r + 3 + 2*j, r - 2},
   	{3, 4*r + 4, 4*r + 1, 0},
   	{2*r + 4, 4*r + 5, 2*r + 1, 0},
   	{r + 3, 5*r + 5, 4*r + 2, 0},
   	{2*r + 5, 6*r + 5, 4*r, 0},
   	{2*r + 3, 6*r + 6, 4*r + 3, 0},
   	{6*r + 4 - j, 6*r + 7 + j, 3 + 2*j, r - 2},
   	{5*r + 4 - j, 7*r + 6 + j, 2*r + 2 + 2*j, r - 2}
   };

eqns = {i1 + i2 == bi2, bi1 + i2 == 2*n + 1};
thmMapping = {"Theorem 4.2 Case n = 4k", "Theorem 4.2 Case n = 4k+1", 
   "Theorem 4.4 Case n = 4k+2", "Theorem 4.4 Case n = 4k+3"};
nMapping = {12 + 4*r, 9 + 4*r, 4*r + 2, 4*r + 3};
rbound = {1, 0, 2, 2};
summary = {};

Off[Solve::svars];

Print["System of Equations:"];
Print[eqns[[1]]];
Print[eqns[[2]]];

For[thm = 1, thm <= 4, thm++,
  Print["**********************************"];
  Print[thmMapping[[thm]]];
  For[i = 1, i <= Length[Ltable[[thm]]], i++,
   For[k = 1, k <= Length[Ltable[[thm]]], k++,
     teq1 = 
      eqns[[1]] /. {i1 -> Ltable[[thm, i, 3]] /. j -> j1, 
        i2 -> Ltable[[thm, k, 3]] /. j -> j2};
     teq2 = 
      eqns[[2]] /. {i2 -> Ltable[[thm, k, 3]] /. j -> j2, 
        n -> nMapping[[thm]]};
     If[thm == 1 || thm == 2,
      teq1 = teq1 /. bi2 -> Ltable[[thm, k, 2]] /. j -> j2;
      teq2 = teq2 /. bi1 -> Ltable[[thm, i, 2]] /. j -> j1;,
      teq1 = teq1 /. bi2 -> Ltable[[thm, k, 2]] + 1 /. j -> j2;
      teq2 = teq2 /. bi1 -> Ltable[[thm, i, 2]] + 1 /. j -> j1;
      ];
     Print["Lines " <> ToString[i] <> " and " <> ToString[k]];
     Print[teq1];
     Print[teq2];
     sol = 
      Solve[teq1 && teq2 && j1 >= 0 && j2 >= 0 && r >= rbound[[thm]] &&
         j1 <= Ltable[[thm, i, 4]] && j2 <= Ltable[[thm, k, 4]], {j1, 
        j2, r}, Integers];
     If[Length[sol] > 0,
      summary = Append[summary, {thm, i, k, sol}];
      Print[sol];,
      Print["No solutions"];
      ];
     Print["----------------------------------"];
     ];
   ];
  ];
Print["**********************************"];
Print["Exception Cases:"];
Print["Theorem, First Line, Second Line, Solution"];
For[i = 1, i <= Length[summary], i++,
  Print[thmMapping[[summary[[i, 1]]]] <> ", " <> 
     ToString[summary[[i, 2]]] <> ", " <> ToString[summary[[i, 3]]] <>
      ", " <> ToString[summary[[i, 4]]]];
  ];
\end{verbatim}


\section{Appendix B: Program Results}
\label{results}
\begin{multicols}{2}
\begin{verbatim}
System of Equations:

i1+i2==bi2

bi1+i2==1+2 n

**********************************

Theorem 4.2 Case n = 4k

Lines 1 and 1

8+2 j1+2 j2==7+j2+2 r

11+j1+2 j2+2 r==1+2 (12+4 r)

No solutions

----------------------------------

Lines 1 and 2

11+2 j1+2 j2+2 r==9+j2+3 r

14+j1+2 j2+4 r==1+2 (12+4 r)

No solutions

----------------------------------

Lines 1 and 3

9+2 j1+2 j2==17+j2+6 r

12+j1+2 j2+2 r==1+2 (12+4 r)

No solutions

----------------------------------

Lines 1 and 4

10+2 j1+2 j2+2 r==18+j2+7 r

13+j1+2 j2+4 r==1+2 (12+4 r)

No solutions

----------------------------------

Lines 1 and 5

13+2 j1+4 r==17+7 r

16+j1+6 r==1+2 (12+4 r)

No solutions

----------------------------------

Lines 1 and 6

16+2 j1+4 r==21+8 r

19+j1+6 r==1+2 (12+4 r)

No solutions

----------------------------------

Lines 1 and 7

11+2 j1+4 r==13+6 r

14+j1+6 r==1+2 (12+4 r)

No solutions

----------------------------------

Lines 1 and 8

15+2 j1+4 r==16+6 r

18+j1+6 r==1+2 (12+4 r)

No solutions

----------------------------------

Lines 1 and 9

12+2 j1+4 r==19+8 r

15+j1+6 r==1+2 (12+4 r)

No solutions

----------------------------------

Lines 1 and 10

9+2 j1+2 r==15+6 r

12+j1+4 r==1+2 (12+4 r)

No solutions

----------------------------------

Lines 1 and 11

14+2 j1+4 r==14+6 r

17+j1+6 r==1+2 (12+4 r)

No solutions

----------------------------------

Lines 2 and 1

11+2 j1+2 j2+2 r==7+j2+2 r

13+j1+2 j2+3 r==1+2 (12+4 r)

No solutions

----------------------------------

Lines 2 and 2

14+2 j1+2 j2+4 r==9+j2+3 r

16+j1+2 j2+5 r==1+2 (12+4 r)

No solutions

----------------------------------

Lines 2 and 3

12+2 j1+2 j2+2 r==17+j2+6 r

14+j1+2 j2+3 r==1+2 (12+4 r)

No solutions

----------------------------------

Lines 2 and 4

13+2 j1+2 j2+4 r==18+j2+7 r

15+j1+2 j2+5 r==1+2 (12+4 r)

No solutions

----------------------------------

Lines 2 and 5

16+2 j1+6 r==17+7 r

18+j1+7 r==1+2 (12+4 r)

No solutions

----------------------------------

Lines 2 and 6

19+2 j1+6 r==21+8 r

21+j1+7 r==1+2 (12+4 r)

No solutions

----------------------------------

Lines 2 and 7

14+2 j1+6 r==13+6 r

16+j1+7 r==1+2 (12+4 r)

No solutions

----------------------------------

Lines 2 and 8

18+2 j1+6 r==16+6 r

20+j1+7 r==1+2 (12+4 r)

No solutions

----------------------------------

Lines 2 and 9

15+2 j1+6 r==19+8 r

17+j1+7 r==1+2 (12+4 r)

No solutions

----------------------------------

Lines 2 and 10

12+2 j1+4 r==15+6 r

14+j1+5 r==1+2 (12+4 r)

No solutions

----------------------------------

Lines 2 and 11

17+2 j1+6 r==14+6 r

19+j1+7 r==1+2 (12+4 r)

No solutions

----------------------------------

Lines 3 and 1

9+2 j1+2 j2==7+j2+2 r

21+j1+2 j2+6 r==1+2 (12+4 r)

{{j1->ConditionalExpression[2 C[1],
    C[1] \[Element] Integers && C[1] >= 2],
  j2->ConditionalExpression[6+2 C[1],
    C[1] \[Element] Integers && C[1] >= 2],
  r->ConditionalExpression[4+3 C[1],
    C[1] \[Element] Integers && C[1] >= 2]}}

----------------------------------

Lines 3 and 2

12+2 j1+2 j2+2 r==9+j2+3 r

24+j1+2 j2+8 r==1+2 (12+4 r)

{{j1->1,j2->0,r->5}}

----------------------------------

Lines 3 and 3

10+2 j1+2 j2==17+j2+6 r

22+j1+2 j2+6 r==1+2 (12+4 r)

No solutions

----------------------------------

Lines 3 and 4

11+2 j1+2 j2+2 r==18+j2+7 r

23+j1+2 j2+8 r==1+2 (12+4 r)

No solutions

----------------------------------

Lines 3 and 5

14+2 j1+4 r==17+7 r

26+j1+10 r==1+2 (12+4 r)

No solutions

----------------------------------

Lines 3 and 6

17+2 j1+4 r==21+8 r

29+j1+10 r==1+2 (12+4 r)

No solutions

----------------------------------

Lines 3 and 7

12+2 j1+4 r==13+6 r

24+j1+10 r==1+2 (12+4 r)

No solutions

----------------------------------

Lines 3 and 8

16+2 j1+4 r==16+6 r

28+j1+10 r==1+2 (12+4 r)

No solutions

----------------------------------

Lines 3 and 9

13+2 j1+4 r==19+8 r

25+j1+10 r==1+2 (12+4 r)

No solutions

----------------------------------

Lines 3 and 10

10+2 j1+2 r==15+6 r

22+j1+8 r==1+2 (12+4 r)

No solutions

----------------------------------

Lines 3 and 11

15+2 j1+4 r==14+6 r

27+j1+10 r==1+2 (12+4 r)

No solutions

----------------------------------

Lines 4 and 1

10+2 j1+2 j2+2 r==7+j2+2 r

22+j1+2 j2+7 r==1+2 (12+4 r)

No solutions

----------------------------------

Lines 4 and 2

13+2 j1+2 j2+4 r==9+j2+3 r

25+j1+2 j2+9 r==1+2 (12+4 r)

No solutions

----------------------------------

Lines 4 and 3

11+2 j1+2 j2+2 r==17+j2+6 r

23+j1+2 j2+7 r==1+2 (12+4 r)

No solutions

----------------------------------

Lines 4 and 4

12+2 j1+2 j2+4 r==18+j2+7 r

24+j1+2 j2+9 r==1+2 (12+4 r)

No solutions

----------------------------------

Lines 4 and 5

15+2 j1+6 r==17+7 r

27+j1+11 r==1+2 (12+4 r)

No solutions

----------------------------------

Lines 4 and 6

18+2 j1+6 r==21+8 r

30+j1+11 r==1+2 (12+4 r)

No solutions

----------------------------------

Lines 4 and 7

13+2 j1+6 r==13+6 r

25+j1+11 r==1+2 (12+4 r)

No solutions

----------------------------------

Lines 4 and 8

17+2 j1+6 r==16+6 r

29+j1+11 r==1+2 (12+4 r)

No solutions

----------------------------------

Lines 4 and 9

14+2 j1+6 r==19+8 r

26+j1+11 r==1+2 (12+4 r)

No solutions

----------------------------------

Lines 4 and 10

11+2 j1+4 r==15+6 r

23+j1+9 r==1+2 (12+4 r)

No solutions

----------------------------------

Lines 4 and 11

16+2 j1+6 r==14+6 r

28+j1+11 r==1+2 (12+4 r)

No solutions

----------------------------------

Lines 5 and 1

13+2 j2+4 r==7+j2+2 r

21+2 j2+7 r==1+2 (12+4 r)

No solutions

----------------------------------

Lines 5 and 2

16+2 j2+6 r==9+j2+3 r

24+2 j2+9 r==1+2 (12+4 r)

No solutions

----------------------------------

Lines 5 and 3

14+2 j2+4 r==17+j2+6 r

22+2 j2+7 r==1+2 (12+4 r)

No solutions

----------------------------------

Lines 5 and 4

15+2 j2+6 r==18+j2+7 r

23+2 j2+9 r==1+2 (12+4 r)

No solutions

----------------------------------

Lines 5 and 5

18+8 r==17+7 r

26+11 r==1+2 (12+4 r)

No solutions

----------------------------------

Lines 5 and 6

True

29+11 r==1+2 (12+4 r)

No solutions

----------------------------------

Lines 5 and 7

16+8 r==13+6 r

24+11 r==1+2 (12+4 r)

No solutions

----------------------------------

Lines 5 and 8

20+8 r==16+6 r

28+11 r==1+2 (12+4 r)

No solutions

----------------------------------

Lines 5 and 9

17+8 r==19+8 r

25+11 r==1+2 (12+4 r)

No solutions

----------------------------------

Lines 5 and 10

14+6 r==15+6 r

22+9 r==1+2 (12+4 r)

No solutions

----------------------------------

Lines 5 and 11

19+8 r==14+6 r

27+11 r==1+2 (12+4 r)

No solutions

----------------------------------

Lines 6 and 1

16+2 j2+4 r==7+j2+2 r

25+2 j2+8 r==1+2 (12+4 r)

No solutions

----------------------------------

Lines 6 and 2

19+2 j2+6 r==9+j2+3 r

28+2 j2+10 r==1+2 (12+4 r)

No solutions

----------------------------------

Lines 6 and 3

17+2 j2+4 r==17+j2+6 r

26+2 j2+8 r==1+2 (12+4 r)

No solutions

----------------------------------

Lines 6 and 4

18+2 j2+6 r==18+j2+7 r

27+2 j2+10 r==1+2 (12+4 r)

No solutions

----------------------------------

Lines 6 and 5

21+8 r==17+7 r

30+12 r==1+2 (12+4 r)

No solutions

----------------------------------

Lines 6 and 6

24+8 r==21+8 r

33+12 r==1+2 (12+4 r)

No solutions

----------------------------------

Lines 6 and 7

19+8 r==13+6 r

28+12 r==1+2 (12+4 r)

No solutions

----------------------------------

Lines 6 and 8

23+8 r==16+6 r

32+12 r==1+2 (12+4 r)

No solutions

----------------------------------

Lines 6 and 9

20+8 r==19+8 r

29+12 r==1+2 (12+4 r)

No solutions

----------------------------------

Lines 6 and 10

17+6 r==15+6 r

26+10 r==1+2 (12+4 r)

No solutions

----------------------------------

Lines 6 and 11

22+8 r==14+6 r

31+12 r==1+2 (12+4 r)

No solutions

----------------------------------

Lines 7 and 1

11+2 j2+4 r==7+j2+2 r

17+2 j2+6 r==1+2 (12+4 r)

No solutions

----------------------------------

Lines 7 and 2

14+2 j2+6 r==9+j2+3 r

20+2 j2+8 r==1+2 (12+4 r)

No solutions

----------------------------------

Lines 7 and 3

12+2 j2+4 r==17+j2+6 r

18+2 j2+6 r==1+2 (12+4 r)

No solutions

----------------------------------

Lines 7 and 4

13+2 j2+6 r==18+j2+7 r

19+2 j2+8 r==1+2 (12+4 r)

No solutions

----------------------------------

Lines 7 and 5

16+8 r==17+7 r

22+10 r==1+2 (12+4 r)

No solutions

----------------------------------

Lines 7 and 6

19+8 r==21+8 r

25+10 r==1+2 (12+4 r)

No solutions

----------------------------------

Lines 7 and 7

14+8 r==13+6 r

20+10 r==1+2 (12+4 r)

No solutions

----------------------------------

Lines 7 and 8

18+8 r==16+6 r

24+10 r==1+2 (12+4 r)

No solutions

----------------------------------

Lines 7 and 9

15+8 r==19+8 r

21+10 r==1+2 (12+4 r)

No solutions

----------------------------------

Lines 7 and 10

12+6 r==15+6 r

18+8 r==1+2 (12+4 r)

No solutions

----------------------------------

Lines 7 and 11

17+8 r==14+6 r

23+10 r==1+2 (12+4 r)

No solutions

----------------------------------

Lines 8 and 1

15+2 j2+4 r==7+j2+2 r

20+2 j2+6 r==1+2 (12+4 r)

No solutions

----------------------------------

Lines 8 and 2

18+2 j2+6 r==9+j2+3 r

23+2 j2+8 r==1+2 (12+4 r)

No solutions

----------------------------------

Lines 8 and 3

16+2 j2+4 r==17+j2+6 r

21+2 j2+6 r==1+2 (12+4 r)

No solutions

----------------------------------

Lines 8 and 4

17+2 j2+6 r==18+j2+7 r

22+2 j2+8 r==1+2 (12+4 r)

No solutions

----------------------------------

Lines 8 and 5

20+8 r==17+7 r

25+10 r==1+2 (12+4 r)

No solutions

----------------------------------

Lines 8 and 6

23+8 r==21+8 r

28+10 r==1+2 (12+4 r)

No solutions

----------------------------------

Lines 8 and 7

18+8 r==13+6 r

23+10 r==1+2 (12+4 r)

No solutions

----------------------------------

Lines 8 and 8

22+8 r==16+6 r

27+10 r==1+2 (12+4 r)

No solutions

----------------------------------

Lines 8 and 9

True

24+10 r==1+2 (12+4 r)

No solutions

----------------------------------

Lines 8 and 10

16+6 r==15+6 r

21+8 r==1+2 (12+4 r)

No solutions

----------------------------------

Lines 8 and 11

21+8 r==14+6 r

26+10 r==1+2 (12+4 r)

No solutions

----------------------------------

Lines 9 and 1

12+2 j2+4 r==7+j2+2 r

23+2 j2+8 r==1+2 (12+4 r)

No solutions

----------------------------------

Lines 9 and 2

15+2 j2+6 r==9+j2+3 r

26+2 j2+10 r==1+2 (12+4 r)

No solutions

----------------------------------

Lines 9 and 3

13+2 j2+4 r==17+j2+6 r

24+2 j2+8 r==1+2 (12+4 r)

No solutions

----------------------------------

Lines 9 and 4

14+2 j2+6 r==18+j2+7 r

25+2 j2+10 r==1+2 (12+4 r)

No solutions

----------------------------------

Lines 9 and 5

17+8 r==17+7 r

28+12 r==1+2 (12+4 r)

No solutions

----------------------------------

Lines 9 and 6

20+8 r==21+8 r

31+12 r==1+2 (12+4 r)

No solutions

----------------------------------

Lines 9 and 7

15+8 r==13+6 r

26+12 r==1+2 (12+4 r)

No solutions

----------------------------------

Lines 9 and 8

19+8 r==16+6 r

30+12 r==1+2 (12+4 r)

No solutions

----------------------------------

Lines 9 and 9

16+8 r==19+8 r

27+12 r==1+2 (12+4 r)

No solutions

----------------------------------

Lines 9 and 10

13+6 r==15+6 r

24+10 r==1+2 (12+4 r)

No solutions

----------------------------------

Lines 9 and 11

18+8 r==14+6 r

29+12 r==1+2 (12+4 r)

No solutions

----------------------------------

Lines 10 and 1

9+2 j2+2 r==7+j2+2 r

19+2 j2+6 r==1+2 (12+4 r)

No solutions

----------------------------------

Lines 10 and 2

12+2 j2+4 r==9+j2+3 r

22+2 j2+8 r==1+2 (12+4 r)

No solutions

----------------------------------

Lines 10 and 3

10+2 j2+2 r==17+j2+6 r

20+2 j2+6 r==1+2 (12+4 r)

No solutions

----------------------------------

Lines 10 and 4

11+2 j2+4 r==18+j2+7 r

21+2 j2+8 r==1+2 (12+4 r)

No solutions

----------------------------------

Lines 10 and 5

14+6 r==17+7 r

24+10 r==1+2 (12+4 r)

No solutions

----------------------------------

Lines 10 and 6

17+6 r==21+8 r

27+10 r==1+2 (12+4 r)

No solutions

----------------------------------

Lines 10 and 7

12+6 r==13+6 r

22+10 r==1+2 (12+4 r)

No solutions

----------------------------------

Lines 10 and 8

True

26+10 r==1+2 (12+4 r)

No solutions

----------------------------------

Lines 10 and 9

13+6 r==19+8 r

23+10 r==1+2 (12+4 r)

No solutions

----------------------------------

Lines 10 and 10

10+4 r==15+6 r

20+8 r==1+2 (12+4 r)

No solutions

----------------------------------

Lines 10 and 11

15+6 r==14+6 r

25+10 r==1+2 (12+4 r)

No solutions

----------------------------------

Lines 11 and 1

14+2 j2+4 r==7+j2+2 r

18+2 j2+6 r==1+2 (12+4 r)

No solutions

----------------------------------

Lines 11 and 2

17+2 j2+6 r==9+j2+3 r

21+2 j2+8 r==1+2 (12+4 r)

No solutions

----------------------------------

Lines 11 and 3

15+2 j2+4 r==17+j2+6 r

19+2 j2+6 r==1+2 (12+4 r)

No solutions

----------------------------------

Lines 11 and 4

16+2 j2+6 r==18+j2+7 r

20+2 j2+8 r==1+2 (12+4 r)

No solutions

----------------------------------

Lines 11 and 5

19+8 r==17+7 r

23+10 r==1+2 (12+4 r)

No solutions

----------------------------------

Lines 11 and 6

22+8 r==21+8 r

26+10 r==1+2 (12+4 r)

No solutions

----------------------------------

Lines 11 and 7

17+8 r==13+6 r

21+10 r==1+2 (12+4 r)

No solutions

----------------------------------

Lines 11 and 8

21+8 r==16+6 r

25+10 r==1+2 (12+4 r)

No solutions

----------------------------------

Lines 11 and 9

18+8 r==19+8 r

22+10 r==1+2 (12+4 r)

No solutions

----------------------------------

Lines 11 and 10

True

19+8 r==1+2 (12+4 r)

No solutions

----------------------------------

Lines 11 and 11

20+8 r==14+6 r

24+10 r==1+2 (12+4 r)

No solutions

----------------------------------

**********************************

Theorem 4.2 Case n = 4k+1

Lines 1 and 1

6+2 j1+2 j2==6+j2+2 r

9+j1+2 j2+2 r==1+2 (9+4 r)

No solutions

----------------------------------

Lines 1 and 2

9+2 j1+2 j2+2 r==8+j2+3 r

12+j1+2 j2+4 r==1+2 (9+4 r)

No solutions

----------------------------------

Lines 1 and 3

7+2 j1+2 j2==14+j2+6 r

10+j1+2 j2+2 r==1+2 (9+4 r)

No solutions

----------------------------------

Lines 1 and 4

8+2 j1+2 j2+2 r==15+j2+7 r

11+j1+2 j2+4 r==1+2 (9+4 r)

No solutions

----------------------------------

Lines 1 and 5

10+2 j1+4 r==14+7 r

13+j1+6 r==1+2 (9+4 r)

No solutions

----------------------------------

Lines 1 and 6

12+2 j1+4 r==17+8 r

15+j1+6 r==1+2 (9+4 r)

No solutions

----------------------------------

Lines 1 and 7

11+2 j1+4 r==12+6 r

14+j1+6 r==1+2 (9+4 r)

No solutions

----------------------------------

Lines 1 and 8

9+2 j1+4 r==11+6 r

12+j1+6 r==1+2 (9+4 r)

No solutions

----------------------------------

Lines 1 and 9

7+2 j1+2 r==13+6 r

10+j1+4 r==1+2 (9+4 r)

No solutions

----------------------------------

Lines 2 and 1

9+2 j1+2 j2+2 r==6+j2+2 r

11+j1+2 j2+3 r==1+2 (9+4 r)

No solutions

----------------------------------

Lines 2 and 2

12+2 j1+2 j2+4 r==8+j2+3 r

14+j1+2 j2+5 r==1+2 (9+4 r)

No solutions

----------------------------------

Lines 2 and 3

10+2 j1+2 j2+2 r==14+j2+6 r

12+j1+2 j2+3 r==1+2 (9+4 r)

No solutions

----------------------------------

Lines 2 and 4

11+2 j1+2 j2+4 r==15+j2+7 r

13+j1+2 j2+5 r==1+2 (9+4 r)

No solutions

----------------------------------

Lines 2 and 5

13+2 j1+6 r==14+7 r

15+j1+7 r==1+2 (9+4 r)

No solutions

----------------------------------

Lines 2 and 6

15+2 j1+6 r==17+8 r

17+j1+7 r==1+2 (9+4 r)

No solutions

----------------------------------

Lines 2 and 7

14+2 j1+6 r==12+6 r

16+j1+7 r==1+2 (9+4 r)

No solutions

----------------------------------

Lines 2 and 8

12+2 j1+6 r==11+6 r

14+j1+7 r==1+2 (9+4 r)

No solutions

----------------------------------

Lines 2 and 9

10+2 j1+4 r==13+6 r

12+j1+5 r==1+2 (9+4 r)

No solutions

----------------------------------

Lines 3 and 1

7+2 j1+2 j2==6+j2+2 r

17+j1+2 j2+6 r==1+2 (9+4 r)

{{j1->ConditionalExpression[2 C[1],
    C[1] \[Element] Integers && C[1] >= 1],
  j2->ConditionalExpression[3+2 C[1],
    C[1] \[Element] Integers && C[1] >= 1],
  r->ConditionalExpression[2+3 C[1],
    C[1] \[Element] Integers && C[1] >= 1]}}

----------------------------------

Lines 3 and 2

10+2 j1+2 j2+2 r==8+j2+3 r

20+j1+2 j2+8 r==1+2 (9+4 r)

No solutions

----------------------------------

Lines 3 and 3

8+2 j1+2 j2==14+j2+6 r

18+j1+2 j2+6 r==1+2 (9+4 r)

No solutions

----------------------------------

Lines 3 and 4

9+2 j1+2 j2+2 r==15+j2+7 r

19+j1+2 j2+8 r==1+2 (9+4 r)

No solutions

----------------------------------

Lines 3 and 5

11+2 j1+4 r==14+7 r

21+j1+10 r==1+2 (9+4 r)

No solutions

----------------------------------

Lines 3 and 6

13+2 j1+4 r==17+8 r

23+j1+10 r==1+2 (9+4 r)

No solutions

----------------------------------

Lines 3 and 7

12+2 j1+4 r==12+6 r

22+j1+10 r==1+2 (9+4 r)

No solutions

----------------------------------

Lines 3 and 8

10+2 j1+4 r==11+6 r

20+j1+10 r==1+2 (9+4 r)

No solutions

----------------------------------

Lines 3 and 9

8+2 j1+2 r==13+6 r

18+j1+8 r==1+2 (9+4 r)

No solutions

----------------------------------

Lines 4 and 1

8+2 j1+2 j2+2 r==6+j2+2 r

18+j1+2 j2+7 r==1+2 (9+4 r)

No solutions

----------------------------------

Lines 4 and 2

11+2 j1+2 j2+4 r==8+j2+3 r

21+j1+2 j2+9 r==1+2 (9+4 r)

No solutions

----------------------------------

Lines 4 and 3

9+2 j1+2 j2+2 r==14+j2+6 r

19+j1+2 j2+7 r==1+2 (9+4 r)

No solutions

----------------------------------

Lines 4 and 4

10+2 j1+2 j2+4 r==15+j2+7 r

20+j1+2 j2+9 r==1+2 (9+4 r)

No solutions

----------------------------------

Lines 4 and 5

12+2 j1+6 r==14+7 r

22+j1+11 r==1+2 (9+4 r)

No solutions

----------------------------------

Lines 4 and 6

14+2 j1+6 r==17+8 r

24+j1+11 r==1+2 (9+4 r)

No solutions

----------------------------------

Lines 4 and 7

13+2 j1+6 r==12+6 r

23+j1+11 r==1+2 (9+4 r)

No solutions

----------------------------------

Lines 4 and 8

11+2 j1+6 r==11+6 r

21+j1+11 r==1+2 (9+4 r)

No solutions

----------------------------------

Lines 4 and 9

9+2 j1+4 r==13+6 r

19+j1+9 r==1+2 (9+4 r)

No solutions

----------------------------------

Lines 5 and 1

10+2 j2+4 r==6+j2+2 r

17+2 j2+7 r==1+2 (9+4 r)

No solutions

----------------------------------

Lines 5 and 2

13+2 j2+6 r==8+j2+3 r

20+2 j2+9 r==1+2 (9+4 r)

No solutions

----------------------------------

Lines 5 and 3

11+2 j2+4 r==14+j2+6 r

18+2 j2+7 r==1+2 (9+4 r)

No solutions

----------------------------------

Lines 5 and 4

12+2 j2+6 r==15+j2+7 r

19+2 j2+9 r==1+2 (9+4 r)

No solutions

----------------------------------

Lines 5 and 5

14+8 r==14+7 r

21+11 r==1+2 (9+4 r)

No solutions

----------------------------------

Lines 5 and 6

16+8 r==17+8 r

23+11 r==1+2 (9+4 r)

No solutions

----------------------------------

Lines 5 and 7

15+8 r==12+6 r

22+11 r==1+2 (9+4 r)

No solutions

----------------------------------

Lines 5 and 8

13+8 r==11+6 r

20+11 r==1+2 (9+4 r)

No solutions

----------------------------------

Lines 5 and 9

11+6 r==13+6 r

18+9 r==1+2 (9+4 r)

No solutions

----------------------------------

Lines 6 and 1

12+2 j2+4 r==6+j2+2 r

20+2 j2+8 r==1+2 (9+4 r)

No solutions

----------------------------------

Lines 6 and 2

15+2 j2+6 r==8+j2+3 r

23+2 j2+10 r==1+2 (9+4 r)

No solutions

----------------------------------

Lines 6 and 3

13+2 j2+4 r==14+j2+6 r

21+2 j2+8 r==1+2 (9+4 r)

No solutions

----------------------------------

Lines 6 and 4

14+2 j2+6 r==15+j2+7 r

22+2 j2+10 r==1+2 (9+4 r)

No solutions

----------------------------------

Lines 6 and 5

16+8 r==14+7 r

24+12 r==1+2 (9+4 r)

No solutions

----------------------------------

Lines 6 and 6

18+8 r==17+8 r

26+12 r==1+2 (9+4 r)

No solutions

----------------------------------

Lines 6 and 7

17+8 r==12+6 r

25+12 r==1+2 (9+4 r)

No solutions

----------------------------------

Lines 6 and 8

15+8 r==11+6 r

23+12 r==1+2 (9+4 r)

No solutions

----------------------------------

Lines 6 and 9

True

21+10 r==1+2 (9+4 r)

No solutions

----------------------------------

Lines 7 and 1

11+2 j2+4 r==6+j2+2 r

15+2 j2+6 r==1+2 (9+4 r)

No solutions

----------------------------------

Lines 7 and 2

14+2 j2+6 r==8+j2+3 r

18+2 j2+8 r==1+2 (9+4 r)

No solutions

----------------------------------

Lines 7 and 3

12+2 j2+4 r==14+j2+6 r

16+2 j2+6 r==1+2 (9+4 r)

No solutions

----------------------------------

Lines 7 and 4

13+2 j2+6 r==15+j2+7 r

17+2 j2+8 r==1+2 (9+4 r)

No solutions

----------------------------------

Lines 7 and 5

15+8 r==14+7 r

19+10 r==1+2 (9+4 r)

No solutions

----------------------------------

Lines 7 and 6

True

21+10 r==1+2 (9+4 r)

No solutions

----------------------------------

Lines 7 and 7

16+8 r==12+6 r

20+10 r==1+2 (9+4 r)

No solutions

----------------------------------

Lines 7 and 8

14+8 r==11+6 r

18+10 r==1+2 (9+4 r)

No solutions

----------------------------------

Lines 7 and 9

12+6 r==13+6 r

16+8 r==1+2 (9+4 r)

No solutions

----------------------------------

Lines 8 and 1

9+2 j2+4 r==6+j2+2 r

14+2 j2+6 r==1+2 (9+4 r)

No solutions

----------------------------------

Lines 8 and 2

12+2 j2+6 r==8+j2+3 r

17+2 j2+8 r==1+2 (9+4 r)

No solutions

----------------------------------

Lines 8 and 3

10+2 j2+4 r==14+j2+6 r

15+2 j2+6 r==1+2 (9+4 r)

No solutions

----------------------------------

Lines 8 and 4

11+2 j2+6 r==15+j2+7 r

16+2 j2+8 r==1+2 (9+4 r)

No solutions

----------------------------------

Lines 8 and 5

13+8 r==14+7 r

18+10 r==1+2 (9+4 r)

No solutions

----------------------------------

Lines 8 and 6

15+8 r==17+8 r

20+10 r==1+2 (9+4 r)

No solutions

----------------------------------

Lines 8 and 7

14+8 r==12+6 r

19+10 r==1+2 (9+4 r)

No solutions

----------------------------------

Lines 8 and 8

12+8 r==11+6 r

17+10 r==1+2 (9+4 r)

No solutions

----------------------------------

Lines 8 and 9

10+6 r==13+6 r

15+8 r==1+2 (9+4 r)

No solutions

----------------------------------

Lines 9 and 1

7+2 j2+2 r==6+j2+2 r

16+2 j2+6 r==1+2 (9+4 r)

No solutions

----------------------------------

Lines 9 and 2

10+2 j2+4 r==8+j2+3 r

19+2 j2+8 r==1+2 (9+4 r)

No solutions

----------------------------------

Lines 9 and 3

8+2 j2+2 r==14+j2+6 r

17+2 j2+6 r==1+2 (9+4 r)

No solutions

----------------------------------

Lines 9 and 4

9+2 j2+4 r==15+j2+7 r

18+2 j2+8 r==1+2 (9+4 r)

No solutions

----------------------------------

Lines 9 and 5

11+6 r==14+7 r

20+10 r==1+2 (9+4 r)

No solutions

----------------------------------

Lines 9 and 6

13+6 r==17+8 r

22+10 r==1+2 (9+4 r)

No solutions

----------------------------------

Lines 9 and 7

True

21+10 r==1+2 (9+4 r)

No solutions

----------------------------------

Lines 9 and 8

10+6 r==11+6 r

19+10 r==1+2 (9+4 r)

No solutions

----------------------------------

Lines 9 and 9

8+4 r==13+6 r

17+8 r==1+2 (9+4 r)

No solutions

----------------------------------

**********************************

Theorem 4.4 Case n = 4k+2

Lines 1 and 1

8+2 j1+2 j2==5+j2+2 r

9+j1+2 j2+2 r==1+2 (2+4 r)

No solutions

----------------------------------

Lines 1 and 2

5+2 j1+2 j2+2 r==4+j2+3 r

6+j1+2 j2+4 r==1+2 (2+4 r)

No solutions

----------------------------------

Lines 1 and 3

7+2 j1+2 j2==5+j2+6 r

8+j1+2 j2+2 r==1+2 (2+4 r)

No solutions

----------------------------------

Lines 1 and 4

6+2 j1+2 j2+2 r==5+j2+7 r

7+j1+2 j2+4 r==1+2 (2+4 r)

No solutions

----------------------------------

Lines 1 and 5

4+2 j1+2 r==4+4 r

5+j1+4 r==1+2 (2+4 r)

No solutions

----------------------------------

Lines 1 and 6

5+2 j1+4 r==4+7 r

6+j1+6 r==1+2 (2+4 r)

No solutions

----------------------------------

Lines 1 and 7

6+2 j1+4 r==4+6 r

7+j1+6 r==1+2 (2+4 r)

No solutions

----------------------------------

Lines 1 and 8

4+2 j1+4 r==3+6 r

5+j1+6 r==1+2 (2+4 r)

No solutions

----------------------------------

Lines 2 and 1

5+2 j1+2 j2+2 r==5+j2+2 r

8+j1+2 j2+3 r==1+2 (2+4 r)

No solutions

----------------------------------

Lines 2 and 2

2+2 j1+2 j2+4 r==4+j2+3 r

5+j1+2 j2+5 r==1+2 (2+4 r)

No solutions

----------------------------------

Lines 2 and 3

4+2 j1+2 j2+2 r==5+j2+6 r

7+j1+2 j2+3 r==1+2 (2+4 r)

No solutions

----------------------------------

Lines 2 and 4

3+2 j1+2 j2+4 r==5+j2+7 r

6+j1+2 j2+5 r==1+2 (2+4 r)

No solutions

----------------------------------

Lines 2 and 5

1+2 j1+4 r==4+4 r

4+j1+5 r==1+2 (2+4 r)

No solutions

----------------------------------

Lines 2 and 6

2+2 j1+6 r==4+7 r

5+j1+7 r==1+2 (2+4 r)

No solutions

----------------------------------

Lines 2 and 7

3+2 j1+6 r==4+6 r

6+j1+7 r==1+2 (2+4 r)

No solutions

----------------------------------

Lines 2 and 8

1+2 j1+6 r==3+6 r

4+j1+7 r==1+2 (2+4 r)

No solutions

----------------------------------

Lines 3 and 1

7+2 j1+2 j2==5+j2+2 r

9+j1+2 j2+6 r==1+2 (2+4 r)

{{j1->ConditionalExpression[2 C[1],
    C[1] \[Element] Integers && C[1] >= 2],
  j2->ConditionalExpression[-2+2 C[1],
    C[1] \[Element] Integers && C[1] >= 2],
  r->ConditionalExpression[3 C[1],
    C[1] \[Element] Integers && C[1] >= 2]}}

----------------------------------

Lines 3 and 2

4+2 j1+2 j2+2 r==4+j2+3 r

6+j1+2 j2+8 r==1+2 (2+4 r)

No solutions

----------------------------------

Lines 3 and 3

6+2 j1+2 j2==5+j2+6 r

8+j1+2 j2+6 r==1+2 (2+4 r)

No solutions

----------------------------------

Lines 3 and 4

5+2 j1+2 j2+2 r==5+j2+7 r

7+j1+2 j2+8 r==1+2 (2+4 r)

No solutions

----------------------------------

Lines 3 and 5

3+2 j1+2 r==4+4 r

5+j1+8 r==1+2 (2+4 r)

No solutions

----------------------------------

Lines 3 and 6

4+2 j1+4 r==4+7 r

6+j1+10 r==1+2 (2+4 r)

No solutions

----------------------------------

Lines 3 and 7

5+2 j1+4 r==4+6 r

7+j1+10 r==1+2 (2+4 r)

No solutions

----------------------------------

Lines 3 and 8

3+2 j1+4 r==3+6 r

5+j1+10 r==1+2 (2+4 r)

No solutions

----------------------------------

Lines 4 and 1

6+2 j1+2 j2+2 r==5+j2+2 r

9+j1+2 j2+7 r==1+2 (2+4 r)

No solutions

----------------------------------

Lines 4 and 2

3+2 j1+2 j2+4 r==4+j2+3 r

6+j1+2 j2+9 r==1+2 (2+4 r)

No solutions

----------------------------------

Lines 4 and 3

5+2 j1+2 j2+2 r==5+j2+6 r

8+j1+2 j2+7 r==1+2 (2+4 r)

No solutions

----------------------------------

Lines 4 and 4

4+2 j1+2 j2+4 r==5+j2+7 r

7+j1+2 j2+9 r==1+2 (2+4 r)

No solutions

----------------------------------

Lines 4 and 5

2+2 j1+4 r==4+4 r

5+j1+9 r==1+2 (2+4 r)

No solutions

----------------------------------

Lines 4 and 6

3+2 j1+6 r==4+7 r

6+j1+11 r==1+2 (2+4 r)

No solutions

----------------------------------

Lines 4 and 7

4+2 j1+6 r==4+6 r

7+j1+11 r==1+2 (2+4 r)

No solutions

----------------------------------

Lines 4 and 8

2+2 j1+6 r==3+6 r

5+j1+11 r==1+2 (2+4 r)

No solutions

----------------------------------

Lines 5 and 1

4+2 j2+2 r==5+j2+2 r

8+2 j2+4 r==1+2 (2+4 r)

No solutions

----------------------------------

Lines 5 and 2

1+2 j2+4 r==4+j2+3 r

5+2 j2+6 r==1+2 (2+4 r)

No solutions

----------------------------------

Lines 5 and 3

3+2 j2+2 r==5+j2+6 r

7+2 j2+4 r==1+2 (2+4 r)

No solutions

----------------------------------

Lines 5 and 4

2+2 j2+4 r==5+j2+7 r

6+2 j2+6 r==1+2 (2+4 r)

No solutions

----------------------------------

Lines 5 and 5

4 r==4+4 r

4+6 r==1+2 (2+4 r)

No solutions

----------------------------------

Lines 5 and 6

1+6 r==4+7 r

5+8 r==1+2 (2+4 r)

No solutions

----------------------------------

Lines 5 and 7

2+6 r==4+6 r

6+8 r==1+2 (2+4 r)

No solutions

----------------------------------

Lines 5 and 8

6 r==3+6 r

4+8 r==1+2 (2+4 r)

No solutions

----------------------------------

Lines 6 and 1

5+2 j2+4 r==5+j2+2 r

8+2 j2+7 r==1+2 (2+4 r)

No solutions

----------------------------------

Lines 6 and 2

2+2 j2+6 r==4+j2+3 r

5+2 j2+9 r==1+2 (2+4 r)

No solutions

----------------------------------

Lines 6 and 3

4+2 j2+4 r==5+j2+6 r

7+2 j2+7 r==1+2 (2+4 r)

No solutions

----------------------------------

Lines 6 and 4

3+2 j2+6 r==5+j2+7 r

6+2 j2+9 r==1+2 (2+4 r)

No solutions

----------------------------------

Lines 6 and 5

1+6 r==4+4 r

4+9 r==1+2 (2+4 r)

No solutions

----------------------------------

Lines 6 and 6

2+8 r==4+7 r

5+11 r==1+2 (2+4 r)

No solutions

----------------------------------

Lines 6 and 7

3+8 r==4+6 r

6+11 r==1+2 (2+4 r)

No solutions

----------------------------------

Lines 6 and 8

1+8 r==3+6 r

4+11 r==1+2 (2+4 r)

No solutions

----------------------------------

Lines 7 and 1

6+2 j2+4 r==5+j2+2 r

8+2 j2+6 r==1+2 (2+4 r)

No solutions

----------------------------------

Lines 7 and 2

3+2 j2+6 r==4+j2+3 r

5+2 j2+8 r==1+2 (2+4 r)

No solutions

----------------------------------

Lines 7 and 3

5+2 j2+4 r==5+j2+6 r

7+2 j2+6 r==1+2 (2+4 r)

No solutions

----------------------------------

Lines 7 and 4

4+2 j2+6 r==5+j2+7 r

6+2 j2+8 r==1+2 (2+4 r)

No solutions

----------------------------------

Lines 7 and 5

2+6 r==4+4 r

4+8 r==1+2 (2+4 r)

No solutions

----------------------------------

Lines 7 and 6

3+8 r==4+7 r

5+10 r==1+2 (2+4 r)

No solutions

----------------------------------

Lines 7 and 7

4+8 r==4+6 r

6+10 r==1+2 (2+4 r)

No solutions

----------------------------------

Lines 7 and 8

2+8 r==3+6 r

4+10 r==1+2 (2+4 r)

No solutions

----------------------------------

Lines 8 and 1

4+2 j2+4 r==5+j2+2 r

7+2 j2+6 r==1+2 (2+4 r)

No solutions

----------------------------------

Lines 8 and 2

1+2 j2+6 r==4+j2+3 r

4+2 j2+8 r==1+2 (2+4 r)

No solutions

----------------------------------

Lines 8 and 3

3+2 j2+4 r==5+j2+6 r

6+2 j2+6 r==1+2 (2+4 r)

No solutions

----------------------------------

Lines 8 and 4

2+2 j2+6 r==5+j2+7 r

5+2 j2+8 r==1+2 (2+4 r)

No solutions

----------------------------------

Lines 8 and 5

6 r==4+4 r

3+8 r==1+2 (2+4 r)

No solutions

----------------------------------

Lines 8 and 6

1+8 r==4+7 r

4+10 r==1+2 (2+4 r)

No solutions

----------------------------------

Lines 8 and 7

2+8 r==4+6 r

5+10 r==1+2 (2+4 r)

No solutions

----------------------------------

Lines 8 and 8

8 r==3+6 r

3+10 r==1+2 (2+4 r)

No solutions

----------------------------------

**********************************

Theorem 4.4 Case n = 4k+3

Lines 1 and 1

8+2 j1+2 j2==7+j2+2 r

11+j1+2 j2+2 r==1+2 (3+4 r)

No solutions

----------------------------------

Lines 1 and 2

7+2 j1+2 j2+2 r==6+j2+3 r

10+j1+2 j2+4 r==1+2 (3+4 r)

No solutions

----------------------------------

Lines 1 and 3

5+2 j1+4 r==5+4 r

8+j1+6 r==1+2 (3+4 r)

No solutions

----------------------------------

Lines 1 and 4

5+2 j1+2 r==6+4 r

8+j1+4 r==1+2 (3+4 r)

No solutions

----------------------------------

Lines 1 and 5

6+2 j1+4 r==6+5 r

9+j1+6 r==1+2 (3+4 r)

No solutions

----------------------------------

Lines 1 and 6

4+2 j1+4 r==6+6 r

7+j1+6 r==1+2 (3+4 r)

No solutions

----------------------------------

Lines 1 and 7

7+2 j1+4 r==7+6 r

10+j1+6 r==1+2 (3+4 r)

No solutions

----------------------------------

Lines 1 and 8

7+2 j1+2 j2==8+j2+6 r

10+j1+2 j2+2 r==1+2 (3+4 r)

No solutions

----------------------------------

Lines 1 and 9

6+2 j1+2 j2+2 r==7+j2+7 r

9+j1+2 j2+4 r==1+2 (3+4 r)

No solutions

----------------------------------

Lines 2 and 1

7+2 j1+2 j2+2 r==7+j2+2 r

10+j1+2 j2+3 r==1+2 (3+4 r)

No solutions

----------------------------------

Lines 2 and 2

6+2 j1+2 j2+4 r==6+j2+3 r

9+j1+2 j2+5 r==1+2 (3+4 r)

No solutions

----------------------------------

Lines 2 and 3

4+2 j1+6 r==5+4 r

7+j1+7 r==1+2 (3+4 r)

No solutions

----------------------------------

Lines 2 and 4

4+2 j1+4 r==6+4 r

7+j1+5 r==1+2 (3+4 r)

No solutions

----------------------------------

Lines 2 and 5

5+2 j1+6 r==6+5 r

8+j1+7 r==1+2 (3+4 r)

No solutions

----------------------------------

Lines 2 and 6

3+2 j1+6 r==6+6 r

6+j1+7 r==1+2 (3+4 r)

No solutions

----------------------------------

Lines 2 and 7

6+2 j1+6 r==7+6 r

9+j1+7 r==1+2 (3+4 r)

No solutions

----------------------------------

Lines 2 and 8

6+2 j1+2 j2+2 r==8+j2+6 r

9+j1+2 j2+3 r==1+2 (3+4 r)

No solutions

----------------------------------

Lines 2 and 9

5+2 j1+2 j2+4 r==7+j2+7 r

8+j1+2 j2+5 r==1+2 (3+4 r)

No solutions

----------------------------------

Lines 3 and 1

5+2 j2+4 r==7+j2+2 r

9+2 j2+4 r==1+2 (3+4 r)

No solutions

----------------------------------

Lines 3 and 2

4+2 j2+6 r==6+j2+3 r

8+2 j2+6 r==1+2 (3+4 r)

No solutions

----------------------------------

Lines 3 and 3

2+8 r==5+4 r

6+8 r==1+2 (3+4 r)

No solutions

----------------------------------

Lines 3 and 4

2+6 r==6+4 r

6+6 r==1+2 (3+4 r)

No solutions

----------------------------------

Lines 3 and 5

3+8 r==6+5 r

7+8 r==1+2 (3+4 r)

No solutions

----------------------------------

Lines 3 and 6

1+8 r==6+6 r

5+8 r==1+2 (3+4 r)

No solutions

----------------------------------

Lines 3 and 7

4+8 r==7+6 r

8+8 r==1+2 (3+4 r)

No solutions

----------------------------------

Lines 3 and 8

4+2 j2+4 r==8+j2+6 r

8+2 j2+4 r==1+2 (3+4 r)

No solutions

----------------------------------

Lines 3 and 9

3+2 j2+6 r==7+j2+7 r

7+2 j2+6 r==1+2 (3+4 r)

No solutions

----------------------------------

Lines 4 and 1

5+2 j2+2 r==7+j2+2 r

10+2 j2+4 r==1+2 (3+4 r)

No solutions

----------------------------------

Lines 4 and 2

4+2 j2+4 r==6+j2+3 r

9+2 j2+6 r==1+2 (3+4 r)

No solutions

----------------------------------

Lines 4 and 3

2+6 r==5+4 r

7+8 r==1+2 (3+4 r)

No solutions

----------------------------------

Lines 4 and 4

2+4 r==6+4 r

7+6 r==1+2 (3+4 r)

No solutions

----------------------------------

Lines 4 and 5

3+6 r==6+5 r

8+8 r==1+2 (3+4 r)

No solutions

----------------------------------

Lines 4 and 6

1+6 r==6+6 r

6+8 r==1+2 (3+4 r)

No solutions

----------------------------------

Lines 4 and 7

4+6 r==7+6 r

9+8 r==1+2 (3+4 r)

No solutions

----------------------------------

Lines 4 and 8

4+2 j2+2 r==8+j2+6 r

9+2 j2+4 r==1+2 (3+4 r)

No solutions

----------------------------------

Lines 4 and 9

3+2 j2+4 r==7+j2+7 r

8+2 j2+6 r==1+2 (3+4 r)

No solutions

----------------------------------

Lines 5 and 1

6+2 j2+4 r==7+j2+2 r

10+2 j2+5 r==1+2 (3+4 r)

No solutions

----------------------------------

Lines 5 and 2

5+2 j2+6 r==6+j2+3 r

9+2 j2+7 r==1+2 (3+4 r)

No solutions

----------------------------------

Lines 5 and 3

3+8 r==5+4 r

7+9 r==1+2 (3+4 r)

No solutions

----------------------------------

Lines 5 and 4

3+6 r==6+4 r

7+7 r==1+2 (3+4 r)

No solutions

----------------------------------

Lines 5 and 5

4+8 r==6+5 r

8+9 r==1+2 (3+4 r)

No solutions

----------------------------------

Lines 5 and 6

2+8 r==6+6 r

6+9 r==1+2 (3+4 r)

No solutions

----------------------------------

Lines 5 and 7

5+8 r==7+6 r

9+9 r==1+2 (3+4 r)

No solutions

----------------------------------

Lines 5 and 8

5+2 j2+4 r==8+j2+6 r

9+2 j2+5 r==1+2 (3+4 r)

No solutions

----------------------------------

Lines 5 and 9

4+2 j2+6 r==7+j2+7 r

8+2 j2+7 r==1+2 (3+4 r)

No solutions

----------------------------------

Lines 6 and 1

4+2 j2+4 r==7+j2+2 r

10+2 j2+6 r==1+2 (3+4 r)

No solutions

----------------------------------

Lines 6 and 2

3+2 j2+6 r==6+j2+3 r

9+2 j2+8 r==1+2 (3+4 r)

No solutions

----------------------------------

Lines 6 and 3

1+8 r==5+4 r

7+10 r==1+2 (3+4 r)

No solutions

----------------------------------

Lines 6 and 4

1+6 r==6+4 r

7+8 r==1+2 (3+4 r)

No solutions

----------------------------------

Lines 6 and 5

2+8 r==6+5 r

8+10 r==1+2 (3+4 r)

No solutions

----------------------------------

Lines 6 and 6

8 r==6+6 r

6+10 r==1+2 (3+4 r)

No solutions

----------------------------------

Lines 6 and 7

3+8 r==7+6 r

9+10 r==1+2 (3+4 r)

No solutions

----------------------------------

Lines 6 and 8

3+2 j2+4 r==8+j2+6 r

9+2 j2+6 r==1+2 (3+4 r)

No solutions

----------------------------------

Lines 6 and 9

2+2 j2+6 r==7+j2+7 r

8+2 j2+8 r==1+2 (3+4 r)

No solutions

----------------------------------

Lines 7 and 1

7+2 j2+4 r==7+j2+2 r

11+2 j2+6 r==1+2 (3+4 r)

No solutions

----------------------------------

Lines 7 and 2

6+2 j2+6 r==6+j2+3 r

10+2 j2+8 r==1+2 (3+4 r)

No solutions

----------------------------------

Lines 7 and 3

4+8 r==5+4 r

8+10 r==1+2 (3+4 r)

No solutions

----------------------------------

Lines 7 and 4

4+6 r==6+4 r

8+8 r==1+2 (3+4 r)

No solutions

----------------------------------

Lines 7 and 5

5+8 r==6+5 r

9+10 r==1+2 (3+4 r)

No solutions

----------------------------------

Lines 7 and 6

3+8 r==6+6 r

7+10 r==1+2 (3+4 r)

No solutions

----------------------------------

Lines 7 and 7

6+8 r==7+6 r

10+10 r==1+2 (3+4 r)

No solutions

----------------------------------

Lines 7 and 8

6+2 j2+4 r==8+j2+6 r

10+2 j2+6 r==1+2 (3+4 r)

No solutions

----------------------------------

Lines 7 and 9

5+2 j2+6 r==7+j2+7 r

9+2 j2+8 r==1+2 (3+4 r)

No solutions

----------------------------------

Lines 8 and 1

7+2 j1+2 j2==7+j2+2 r

12+j1+2 j2+6 r==1+2 (3+4 r)

{{j1->ConditionalExpression[1+2 C[1],
    C[1] \[Element] Integers && C[1] >= 4],
  j2->ConditionalExpression[-4+2 C[1],
    C[1] \[Element] Integers && C[1] >= 4],
  r->ConditionalExpression[-1+3 C[1],
    C[1] \[Element] Integers && C[1] >= 4]}}

----------------------------------

Lines 8 and 2

6+2 j1+2 j2+2 r==6+j2+3 r

11+j1+2 j2+8 r==1+2 (3+4 r)

No solutions

----------------------------------

Lines 8 and 3

4+2 j1+4 r==5+4 r

9+j1+10 r==1+2 (3+4 r)

No solutions

----------------------------------

Lines 8 and 4

4+2 j1+2 r==6+4 r

9+j1+8 r==1+2 (3+4 r)

No solutions

----------------------------------

Lines 8 and 5

5+2 j1+4 r==6+5 r

10+j1+10 r==1+2 (3+4 r)

No solutions

----------------------------------

Lines 8 and 6

3+2 j1+4 r==6+6 r

8+j1+10 r==1+2 (3+4 r)

No solutions

----------------------------------

Lines 8 and 7

6+2 j1+4 r==7+6 r

11+j1+10 r==1+2 (3+4 r)

No solutions

----------------------------------

Lines 8 and 8

6+2 j1+2 j2==8+j2+6 r

11+j1+2 j2+6 r==1+2 (3+4 r)

No solutions

----------------------------------

Lines 8 and 9

5+2 j1+2 j2+2 r==7+j2+7 r

10+j1+2 j2+8 r==1+2 (3+4 r)

No solutions

----------------------------------

Lines 9 and 1

6+2 j1+2 j2+2 r==7+j2+2 r

11+j1+2 j2+7 r==1+2 (3+4 r)

{{j1->0,j2->1,r->6}}

----------------------------------

Lines 9 and 2

5+2 j1+2 j2+4 r==6+j2+3 r

10+j1+2 j2+9 r==1+2 (3+4 r)

No solutions

----------------------------------

Lines 9 and 3

3+2 j1+6 r==5+4 r

8+j1+11 r==1+2 (3+4 r)

No solutions

----------------------------------

Lines 9 and 4

3+2 j1+4 r==6+4 r

8+j1+9 r==1+2 (3+4 r)

No solutions

----------------------------------

Lines 9 and 5

4+2 j1+6 r==6+5 r

9+j1+11 r==1+2 (3+4 r)

No solutions

----------------------------------

Lines 9 and 6

2+2 j1+6 r==6+6 r

7+j1+11 r==1+2 (3+4 r)

No solutions

----------------------------------

Lines 9 and 7

5+2 j1+6 r==7+6 r

10+j1+11 r==1+2 (3+4 r)

No solutions

----------------------------------

Lines 9 and 8

5+2 j1+2 j2+2 r==8+j2+6 r

10+j1+2 j2+7 r==1+2 (3+4 r)

No solutions

----------------------------------

Lines 9 and 9

4+2 j1+2 j2+4 r==7+j2+7 r

9+j1+2 j2+9 r==1+2 (3+4 r)

No solutions

----------------------------------

**********************************

Exception Cases:

Theorem, First Line, Second Line, Solution

Theorem 4.2 Case n = 4k, 3, 1, 
{{j1 -> ConditionalExpression[2 C[1], 
    C[1] \[Element] Integers && C[1] >= 2], 
  j2 -> ConditionalExpression[6 + 2 C[1], 
    C[1] \[Element] Integers && C[1] >= 2], 
  r -> ConditionalExpression[4 + 3 C[1], 
    C[1] \[Element] Integers && C[1] >= 2]}}

Theorem 4.2 Case n = 4k, 3, 2, 
{{j1 -> 1, j2 -> 0, r -> 5}}

Theorem 4.2 Case n = 4k+1, 3, 1, 
{{j1 -> ConditionalExpression[2 C[1], 
    C[1] \[Element] Integers && C[1] >= 1], 
  j2 -> ConditionalExpression[3 + 2 C[1], 
    C[1] \[Element] Integers && C[1] >= 1], 
  r -> ConditionalExpression[2 + 3 C[1], 
    C[1] \[Element] Integers && C[1] >= 1]}}

Theorem 4.4 Case n = 4k+2, 3, 1, 
{{j1 -> ConditionalExpression[2 C[1], 
    C[1] \[Element] Integers && C[1] >= 2], 
  j2 -> ConditionalExpression[-2 + 2 C[1], 
    C[1] \[Element] Integers && C[1] >= 2], 
  r -> ConditionalExpression[3 C[1], 
    C[1] \[Element] Integers && C[1] >= 2]}}

Theorem 4.4 Case n = 4k+3, 8, 1, 
{{j1 -> ConditionalExpression[1 + 2 C[1], 
    C[1] \[Element] Integers && C[1] >= 4], 
  j2 -> ConditionalExpression[-4 + 2 C[1], 
    C[1] \[Element] Integers && C[1] >= 4], 
  r -> ConditionalExpression[-1 + 3 C[1], 
    C[1] \[Element] Integers && C[1] >= 4]}}

Theorem 4.4 Case n = 4k+3, 9, 1, 
{{j1 -> 0, j2 -> 1, r -> 6}}
\end{verbatim}
\end{multicols}
\normalsize

\end{document}